\documentclass{article}
\usepackage{graphicx}
\usepackage[colorlinks=true]{hyperref}
\usepackage{amsmath,amsfonts,amssymb,amsthm,comment,aliascnt,xcolor,nicefrac}
\usepackage{hyperref,cleveref,bbm,enumerate,enumitem}

\newcommand{\ooint}[1]{\left(#1\right)}

\newcommand{\inv}{{{\pi}}}

\newcommand{\Prox}{\mathrm{Prox}}
\newcommand{\wh}[1]{{\widehat{#1}}}

\newcommand{\ul}[1]{{\underline{#1}}}

\def\argmin{\mathop\mathrm{argmin}}

\def\la{{\langle}}
\def\ra{{\rangle}}

\def\PP{{{P}}}

\def\rank{{\hbox{\rm rank}}}

\def\F{{\cal F}}

\def\bR{{\mathbf{R}}}
\def\bZ{{\mathbf{Z}}}

\def\A{\mathsf{A}}

\def\cL{{\cal L}}

\def\P{{\cal P}}

\def\bE{{\mathbf{E}}}

\def\Prob{\hbox{\rm Prob}}
\def\Proj{\hbox{\rm Proj}}

\def\Col{\hbox{\rm Col}}

\def\cT{\mathcal{T}}
\newtheorem{corollary}{Corollary}[section]
\newtheorem{proposition}{Proposition}[section]

\newtheorem{theorem}{Theorem}[section]
\newtheorem{lemma}{Lemma}[section]

\def\qed{\hfill$\Box$}
\def\eqed{\eqno{\mbox{\qed}}}

\def\Tr{{\hbox{\rm Tr}}}

\newcommand{\be}{\begin{eqnarray}}
\newcommand{\ee}[1]{\label{#1}\end{eqnarray}}
\newcommand{\nn}{\nonumber \\}

\newcommand{\rf}[1]{(\ref{#1})}
\newcommand{\half}{ \mbox{\small$\frac{1}{2}$}}

\newcommand{\bse}{\begin{eqnarray*}}
\newcommand{\ese}{\end{eqnarray*}}

\definecolor{MyDarkBlue}{rgb}{0,0.08,0.45}
\definecolor{MyViolet}{rgb}{0.45,0.08,0.95}
\definecolor{MyBrown}{rgb}{0.45,0.08,0}

\newcommand{\hide}[1]{{}}

\newcommand{\aic}[2]{{\color{blue}\  #2}}

\newcommand{\ov}[1]{{\overline#1}}

\def\veps{{\varepsilon}}

\newcommand{\norm}[1]{\|#1\|}
\newcommand{\normd}[1]{\|#1\|_*}
\newcommand{\normA}[1]{\|#1\|_Q}
\newcommand{\suite}[2]{\left\{ #1, #2\right\}}

\def\eqdef{:=}
\newcommand{\inner}[2]{\langle #1,#2\rangle}
\def\Xset{\mathsf{X}}
\def\Yset{\mathsf{Y}}
\def\Zset{\mathsf{Z}}
\def\sN{\mathsf{N}}
\def\sM{\mathsf{M}}

\def\P{\mathsf{P}}

\def\mr{{\mathfrak{r}}}
\def\ovs{{\overline{s}}}
\def\rmd{\mathsf{d}}
\def\KK{\mathsf{K}}
\def\AA{\mathsf{A}}
\def\BB{\mathsf{B}}
\def\CC{\mathsf{C}}

\theoremstyle{definition}
\newtheorem{assumption}{}
\renewcommand{\theassumption}{A\arabic{assumption}}

\newenvironment{compassumption}
  {\refstepcounter{assumption}%
   \par\addvspace{\medskipamount}%
   \begin{enumerate}[label=\textbf{\theassumption\alph*}, ref=\theassumption\alph*, leftmargin=*]}
  {\end{enumerate}\par\addvspace{\medskipamount}}

\theoremstyle{definition}
\newtheorem{assumptionR}{}
\renewcommand{\theassumptionR}{R\arabic{assumptionR}}

\theoremstyle{definition}
\newtheorem{assumptionF}{}
\renewcommand{\theassumptionF}{F\arabic{assumptionF}}

\newenvironment{compassumptionF}
  {\refstepcounter{assumptionF}%
   \par\addvspace{\medskipamount}%
   \begin{enumerate}[label=\textbf{\theassumptionF\alph*}, ref=\theassumptionF\alph*, leftmargin=*]}
  {\end{enumerate}\par\addvspace{\medskipamount}}

\theoremstyle{definition}
\newtheorem{assumptionS}{}
\renewcommand{\theassumptionS}{S\arabic{assumptionS}}

\theoremstyle{definition}
\newtheorem{assumptionSR}{}
\renewcommand{\theassumptionSR}{SR\arabic{assumptionSR}}

\Crefname{assumptionP}{P\hspace{-3pt}}{P\hspace{-3pt}}
\crefname{assumptionP}{P}{P}

\Crefname{assumptionFAR}{FAR\hspace{-3pt}}{FAR\hspace{-3pt}}
\crefname{assumptionFAR}{FAR}{FAR}

\Crefname{assumptionVFAR}{VFAR\hspace{-3pt}}{VFAR\hspace{-3pt}}
\crefname{assumptionVFAR}{VFAR}{VFAR}

\Crefname{assumptionDSDE}{noname\hspace{-3pt}}{noname\hspace{-3pt}}
\crefname{assumptionDSDE}{noname}{noname}

\newcommand{\sparse}{\mathrm{sparse}}

\title{Stochastic Mirror Descent under Dependent Observations}
\author{Gersende Fort
\thanks{\scriptsize LAAS-CNRS, Universite de Toulouse, CNRS, Toulouse, France, {\tt gersende.fort@laas.fr}}
\and
Anatoli Juditsky
\thanks{\scriptsize LJK, Universit\'e Grenoble Alpes, Domaine Universitaire de Saint-Martin-d'Hères, France
{\tt anatoli.juditsky@univ-grenoble-alpes.fr}}
\and Jiaming Liang
\thanks{\scriptsize Goergen Institute for Data Science and Artificial Intelligence and Department of Computer Science, University
of Rochester, Rochester, USA, {\tt jiaming.liang@rochester.edu}. This work was partially supported AFOSR grant FA955025-1-0182. }
}
\date{}
\begin{document}
\maketitle
\begin{abstract}
We study the non-asymptotic convergence of Stochastic Mirror Descent (SMD) algorithm when the noisy first-order oracle is driven by dependent observations. Departing from the classical i.i.d. assumption underlying most existing guarantees, we introduce a flexible ``forgetting" condition that accommodates geometrically or polynomially dependent sequences. Under this condition we derive explicit finite-sample accuracy bounds for the averaged SMD iterates, recovering the classical $O(1/T)$ and $O(1/\sqrt{T})$ rates in the large- and small-initial-error regimes. We instantiate the general framework on two applications: high-dimensional linear regression with dependent regressors and nonparametric estimation of the drift function of a functional autoregressive model.
\end{abstract}
\section{Introduction}
In this paper, we discuss Stochastic Mirror Descent (SMD) algorithm  as applied to the stochastic optimization problem of the form
\[
f_*=\min_{x\in\Xset} \{f(x) = \bE[F(x,\omega)] \}
\]
where $\Xset$ is a given convex and compact subset of a Euclidean space $E$, $F(\cdot, \cdot):\,\Xset\times \Yset\to \bR$ is a convex mapping with respect to $x$ for all $\omega$, and $\bE[\cdot]$ stands for the expectation with respect to unknown distribution $\pi$ of $\omega\in \Yset$ (the corresponding expectation is assumed to exist for all $x\in \Xset$).
We suppose that the problem is solvable, and when looking for a solution, we can access a stochastic ``oracle'' which, given an input, a ``search point'' $x\in\Xset$ supplies ``randomized'' information about $f$ at $x$. In this situation, we naturally turn to stochastic first-order algorithms to build recursively approximate solutions to the problem. We choose SMD
 as our workhorse, owing to its scalability, cheap per-iteration cost, and robustness to noisy and biased information about the objective. Since the seminal work of  Nemirovski and Yudin \cite{NemirovskiYudin1983}, a substantial body of work has developed a comprehensive theory of stochastic first-order methods; see, e.g., \cite{BeckTeboulle2003,nemirovski2009robust,juditsky2011optimization,lan2020first} and references therein.
The classical analysis of SMD, in essentially all of its variants, rests on the assumption that the noisy first-order oracle delivers, at each iteration $x_t$ of the method, an unbiased (or asymptotically negligibly biased) estimate
$G(x_t,\omega_{t+1})$ of the gradient $g(x_t)$ of the objective under independent random noises $\omega_{t+1}$.
This is the setting which underlies most textbook non-asymptotic guarantees for SMD. In many applications, however, the observations available to the algorithm are generated by a dependent, e.g., Markovian process, and the independence assumption fails outright. Prominent examples include: on-policy reinforcement learning and temporal-difference (TD) learning, where the data stream is generated by a single trajectory of a Markov Decision Process \cite{BhandariRussoSingal2018,SrikantYing2019,DoanMaguluriRomberg2019}; stochastic-gradient MCMC and Langevin-type samplers, where the ``noise'' itself is produced by an ergodic Markov chain \cite{KarimiMiasojedowMoulinesWai2019}; and, as in the application we develop below, online estimation of  functional autoregressive model from a single realization of a (dependent) time series \cite{juditsky2000minimax}.

It should be mentioned that the analysis of Stochastic Approximation (SA) under dependent or Markovian noise has a long history, going back to the classical works on SA under mixing conditions (e.g., \cite{KushnerYin2003,Benveniste1990}), which typically establish asymptotic
 (a.s. or in-distribution) convergence via ODE-method arguments, but do not provide finite-time, non-asymptotic rates of the type
 available in the independent case (for the ``classical framework'' of finite-time analysis of SA under dependent noises, see, e.g., \cite{nemirovskii1981optimization,shilman1981stochastic}).  Renewed interest in non-asymptotic guarantees for
 Markov-noise SA has been driven largely by the reinforcement-learning literature which led to general Markov Chain SA in \cite{KarimiMiasojedowMoulinesWai2019,DebavelaereAllassonniereDurmus2021}.
 A parallel and closely related line of
 work in the online learning and stochastic optimization community studies mixing/ergodic extensions of Mirror Descent and dual averaging, notably \cite{juditsky2000minimax} and \cite{DuchiAgarwalJohanssonJordan2012}, which develop ergodic mirror descent under general $\phi$-mixing assumptions on the noise sequence.

 The present paper contributes to this line of work in several directions. First, we develop a general, non-asymptotic analysis of SMD under a flexible dependence/``forgetting" condition (Assumption~\ref{hyp:oracle:Gt:mixing}) that covers, as special cases, both i.i.d. observations and geometrically or polynomially mixing dependent sequences, without requiring the observation process itself to be Markovian. Second, the refined analysis allows providing explicit and new, to our knowledge, accuracy bounds. In particular, we show that, similar to the case of independent noises \cite{juditsky2020sparse,ilandarideva2025accelerated}, the SMD algorithm under weak dependence assumption exhibits several ``regimes" of convergence depending on the relationship of the initial distance to the optimal solution and the scale of the asymptotic noise (noise ``at the optimal solution"). Specifically (see, e.g., Proposition~\ref{pro:bgamm1}), the expected sub-optimality $\bE[f(\ov x_T)]-f_*$ of the approximate averaged solution $\ov x_T$ after $T$ iterations decrease as $O(1/T)$ when the initial distance is ``large," to attain the ``asymptotic rate'' $O(1/\sqrt{T})$ when this distance is ``small." Interestingly, a new ``intermediate" regime which is not present in the independent case, appears under  dependent noises.

When the optimization problem admits a sparse solution $x_*$ (e.g., standard, group, or low-rank sparsity), the accuracy of plain SMD can be improved considerably by a multistage scheme. We consider the SMD for Sparse Recovery (SMD-SR) algorithm, an analogue of the multistage routine of \cite{ilandarideva2023stochastic} for the present situation. The method runs SMD with constant stepsize over successive stages followed by the  ``sparsification" of the resulting approximate solution at the end of each stage.  The observation complexity $T^*(\upsilon)$ (number of calls to the stochastic oracle) needed to achieve the expected  accuracy $\upsilon$ of the approximate solution is evaluated in Proposition \ref{prop:SR}.
As in the independent case, the stages of SMD-SR fall into distinct regimes. During the preliminary phase, the stage length grows slowly with the stage count (logarithmically, in case of geometrical forgetting), so the error decreases almost linearly as the total number of oracle calls grows. The intermediate phase, which does not exist in the independent case, follows and shows an $\mathcal{O}(1/T^2)$-type convergence for the expected squared error. Finally, the asymptotic phase gives the standard $\mathcal{O}(1/T)$ rate, of the same order as in the i.i.d. case.

 We also instantiate our machinery on two problems of independent interest: high-dimensional linear regression with dependent (mixing) design, and nonparametric estimation of the drift/regression function of a vector functional autoregressive model.

\paragraph{Organization.} Section \ref{sec:probs} introduces the stochastic optimization problem, the stochastic oracle, and the assumptions under which the principal results of the paper---Propositions \ref{pro:sum1} and \ref{pro:bgamm1}---are derived. We present SMD and state our main convergence results in Section~\ref{sec:SMDproof}; properties of the multistage SMD for sparse optimization are studied in Section~\ref{sec:multis}. Section~\ref{sec:appl} is dedicated to applications of the general framework to the motivating problems described above. All proofs are postponed to the appendices.

\paragraph{Notation.} Let $E$ be a finite dimensional real vector (Euclidean) space. Given a norm $\norm{\cdot}$ on $E$, the associated dual norm $\normd{\cdot}$ is defined as $\normd{z}\eqdef \sup\{\inner{x}{z} : \norm{x}\leq 1\}$. Vectors are column-vectors. For a matrix $\A$, $\A^\top$ denotes its transpose.

For $Q\in \bR^{p\times q}$, we denote
\[\norm{Q}_\infty \eqdef \max_{ij}|Q_{ij}|;
\] for positive-definite $Q\in \bR^{n\times n}$ and $x\in \bR^n$ we denote
\[\norm{x}_Q \eqdef \sqrt{x^\top Q x}.\]
For matrices $x, z\in \bR^{p \times q}$ we  consider the Frobenius inner product $\inner{x}{z}\eqdef\Tr(x^\top z)=\sum_{i=1}^{p}\sum_{j=1}^{q} x_{ij}z_{ij}$. We denote $\norm{x}_{2,1}$ the norm
\[
\norm{x}_{2,1} \eqdef \sum_{i=1}^{q}\left\|\Col_i[x]\right\|_2,
\] where $\Col_i[x]$ stands for the $i$-th column of $x$, and denote $\norm{z}_{2,\infty}$ the conjugate norm,
\[\norm{z}_{2,\infty} \eqdef \max_{1 \leq i\leq q}  \norm{\Col_i[z]}_2.\]
In what follows, we use a generic notation $c$ and $C$ for absolute constants; notation $a \lesssim b$ means that the ratio $a/b$ is bounded by an absolute constant. Unless stated otherwise, all relations between random variables are assumed to hold a.s.

\section{Stochastic Optimization Problem}\label{sec:probs}
 Consider a Stochastic Optimization problem
	\be
	\min_{x\in \Xset}  f(x), \qquad f(x) \eqdef  \bE_{\omega\sim \inv}\left[F(x,\omega)\right],
	\ee{prob1}
where  $\Xset\subset E$ is a convex and compact
set with nonempty interior,\footnote{Otherwise, we may replace $E$ with the affine span of $\Xset$.} $\omega$ is a random variable defined on a probability space $(\Omega, \mathcal{A}, \PP)$ taking values in $\Yset$; and $F:\; \Xset \times \Yset\to \bR$ is measurable.
We denote $D$ the $\|\cdot\|$-diameter of $\Xset$:
\[
D \eqdef \max_{x,x'\in \Xset}\|x-x'\|.
\]

On the objective function, it is assumed:
\begin{assumption} \label{hyp:problem}
For all $\omega$, $ F(\cdot,\omega)$ is convex on $\Xset$;  $f$ is continuously differentiable on $\Xset$; we denote $g(x)$ its gradient at $x\in\Xset$.
  \end{assumption}

Note that in this case, the problem in \rf{prob1} is solvable (the set $\mathrm{Argmin}_\Xset f$ is non-empty); we denote  $f_* = \min_{x\in \Xset} f(x)$ its optimal value. For the sake of simplicity, we suppose that:
\begin{assumption}\label{hyp:g:xstar} There is $x_* \in \mathrm{Argmin}_\Xset f$ such that $g(x_*)=0$.
 \end{assumption}

To solve \rf{prob1}, we use a stochastic algorithm which, at iteration $(t+1)$, given the current iterate $x_t$, accesses a {\em stochastic oracle} of $g(x_t)$ denoted by  $G(x_t,\omega_{t+1})$. Let $\Xset^o$ be an open set containing $\Xset$; we assume that $G: \Xset^o \times \Yset \to E$ is a measurable function and that $\omega_{t+1} \in \Yset$ is a random variable, defined on $(\Omega, \mathcal{A}, \PP)$, whose distribution is not necessarily $\inv$.

Let $ \F \eqdef \suite{\F_t}{ t \geq 0}$  be a nondecreasing filtration such that $x_t \in \F_t$ and $\omega_{t+1} \in \F_{t+1}$ for all $t \geq 0$. We denote $\bE_t\{\cdot\}$ the conditional expectation w.r.t. $\F_t$.

\begin{compassumption} \label{hyp:oracle:decomp}
\item  \label{hyp:oracle:decomp:cocoercivity} %
For each $t \geq 0$, we define an $\F_t$-measurable random mapping $G_t(\cdot)$ on $\Xset^0$
which is  a gradient of a (random) convex function such that $G_t(x_*)=0$ for all $x_*\in \Xset_*$.

\item \label{hyp:oracle:decomp:xi}   Denote $\xi_{t+1} \eqdef G(x_t,\omega_{t+1}) - \bE_t\left[ G(x_t,\omega_{t+1})\right]$.  There exist (nonrandom) constants $\varsigma_* \geq 0$  and $\kappa\geq 0$, such that for all $t \geq 0$,
\be
\bE_{t}[\normd{\xi_{t+1}}^2]\leq  \varsigma_*^2 + \kappa \|G_t(x_t)\|_*^2.
\ee{sigma*}
\item \label{hyp:oracle:decomp:delta} There exists $\delta \geq 0$ such that for all $x\in\Xset$,
\be\sup_{t \geq 0} \bE\left[\normd{\bE_t\left[ G(x,\omega_{t+1})\right]-G_t(x)}^2\right]\leq \delta^2.\ee{deltab}
  \end{compassumption}
We also assume that
\begin{compassumption} \label{hyp:oracle:Gt}
%  \begin{enumerate}[label=\alph*)]
    \item \label{hyp:oracle:Gt:lipschitz} $x \mapsto G_t(x)$ %(or equivalently, $x \mapsto \bE_t\left[ G(x, \omega_{t+1}) \right]$)
    is Lipschitz continuous on $\Xset^o$ (uniformly in $t$): there is $\cL<\infty$  such that for all $t \geq 0$ and $x,x'\in \Xset^o$
\be
\normd{G_t(x)- G_t(x')} \leq  \cL \norm{x-x'}.
\ee{lip1}
\item \label{hyp:oracle:Gt:ell} there exists $1\leq \ell<\infty$ such that for all $t \geq 0$ and $x,x' \in \Xset$
\be
\la G_t(x)- G_t(x'),x-x_*\ra\leq \ell \, \normd{G_t(x)} \, \norm{x-x'}.
\ee{ellsym1}
\item \label{hyp:oracle:Gt:mixing} there exist a positive sequence $\{\rho_s, s \geq 0\}$ such that $\lim_{s \to \infty} \rho_s =0$,  and two families of nonnegative $\F$-adapted random variables $\suite{\sM_t}{ t \geq 0}$ and $\suite{\sN_t}{ t \geq 0}$ such that  $ \ov M^2 \eqdef \sup_t \bE[\sM_t^2]< \infty $, $\ov N^2 \eqdef \sup_t \bE[\sN_t^2]< \infty $, and for all $t,s\geq 0$  and $x \in \Xset$ (cf. \cite{mcleish1974dependent}),
\be \normd{\bE_t[G_{t+s}(x)]- g(x)} \leq  \rho_s  (\sM_t  + \sN_t\norm{x-x_*}).
\ee{mix1}
\end{compassumption}

\paragraph{Remarks.}\begin{enumerate}

\item
Note  that \ref{hyp:oracle:Gt:lipschitz} and  \ref{hyp:oracle:Gt:mixing} imply that the gradient $g(\cdot)$
	is Lipschitz-continuous on $\Xset^o$:
	\[
	\normd{g(x')-g(x)} \leq  \cL \norm{x'-x},\qquad \forall \,x,x'\in \Xset.
	\]
Indeed, for all $x,x' \in \Xset^o$  we have
\begin{align*}
\normd{g(x') - g(x)} &\leq \normd{g(x') - \bE_t\{G_{t+s}(x')\}}  + \normd{\bE_t\{G_{t+s}(x')- G_{t+s}(x)\}}  + \normd{\bE_t\{G_{t+s}(x)\} - g(x)}  \\
& \leq  \rho_s \left( 2\sM_t  + \sN_t \norm{x'-x_*} + \sN_t \norm{x-x_*} \right) +  \cL \norm{x'-x},
\end{align*}
and the first term on the right-hand side tends to zero when $s \to \infty$.
\item   By \ref{hyp:oracle:decomp:cocoercivity} and \ref{hyp:oracle:Gt:lipschitz},  $G_t(\cdot)$ is Lipschitz-continuous on $\Xset^o$ and the gradient of a convex function, so it is cocoercive on $\Xset$; thus, due to $G_t(x_*) =0$,
\[
\normd{G_t(x_t)}^2\leq \cL\la G_t(x_t),x_t-x_*\ra,
\]
and so, by \ref{hyp:oracle:decomp:xi},
\be
\bE_{t}[\normd{\xi_{t+1}}^2]\leq  \varsigma_*^2 + \kappa \|G_t(x_t)\|_*^2\leq
\varsigma_*^2 + \kappa\cL \la G_t(x_t),x_t-x_*\ra.
\ee{cocoxi}
The bound \rf{cocoxi} is operational when establishing convergence of the SMD algorithm in Section \ref{sec:SMDproof}.

Now, suppose that $G(\cdot,\omega)$  is a Lipschitz-continuous  gradient field of a convex function on $\Xset^o$ (e.g., of $F(\cdot,\omega)$) i.e., for some (nonrandom) $\ov\cL<\infty$, almost all $\omega\in \Yset$, and all $x,x' \in \Xset^o$
\begin{equation} \label{eq:Gisgdtcvx}
    \normd{ G(x,\omega)-G(x',\omega)} \leq \ov\cL \norm{x-x'}.
\end{equation}
It is then readily seen that, when setting $G_t(x) \eqdef \bE_t[G(x,\omega_{t+1})]$ and assuming that $G_t(x_*)=0$, \ref{hyp:oracle:decomp:cocoercivity} holds, and \ref{hyp:oracle:Gt:lipschitz} holds with  $\cL \eqdef \ov \cL$: for $t \geq 0$ and $x,x' \in \Xset^0$
\[
\normd{G_t(x) - G_t(x')} \le \bE_t\left[ \normd{G(x, \omega_{t+1}) - G(x', \omega_{t+1})} \right] \leq \ov\cL \norm{x-x'}.
\]
In addition,  we have
\begin{equation}\label{eq:xiGgdtcvx}
\bE_t[\normd{\xi_{t+1}}^2] \leq6\ov\cL\la G_t(x_t),x_t-x_*\ra    + \underbrace{3\bE_t[\normd{G(x_*,\omega_{t+1})}^2]}_{=: \varsigma_t^2},
\end{equation}
where $\varsigma_t^2$ is a nonnegative random variable. Indeed, by construction, \begin{align*}
\xi_{t+1}=G(x_t,\omega_{t+1})-\bE_t[G(x_t,\omega_{t+1})]&=(G(x_t,\omega_{t+1})-G(x_*,\omega_{t+1}))+(G(x_*,\omega_{t+1})-\bE_t[G(x_*,\omega_{t+1})])\\&\quad  +(\bE_t[G(x_*,\omega_{t+1})]-\bE_t[G(x_t,\omega_{t+1})]).
\end{align*}
Because $G(\cdot,\omega)$ is cocoercive on $\Xset$ due to Lipschitz continuity, and $G_t(x_*) =0$, we have, for $x_t\in \Xset$,
\begin{align*}
\normd{\xi_{t+1}}^2&\leq 3 \normd{G(x_t,\omega_{t+1})-G(x_*,\omega_{t+1})}^2+3\normd{G(x_*,\omega_{t+1})}^2+3\normd{G_t(x_t)}^2\\
&\leq 3\ov\cL\la G(x_t,\omega_{t+1})-G(x_*,\omega_{t+1}),x_t-x_*\ra+3\normd{G(x_*,\omega_{t+1})}^2 \\
& +3\ov\cL\la G_t(x_t),x_t-x_*\ra;
\end{align*}
taking the conditional expectation and using the definition of $G_t(x_t)$  yields \eqref{eq:xiGgdtcvx}. Thus, in this case,
\[\bE_t[\normd{\xi_{t+1}}^2] \leq6\ov\cL\la G_t(x_t),x_t-x_*\ra    + \varsigma_*^2\]
with $\varsigma_*^2\geq \sup_t \varsigma_t^2$.
\end{enumerate}

\section{Stochastic Mirror Descent}
\label{sec:SMDproof}
\subsection{Stochastic Mirror Descent algorithm}
\paragraph{Proximal setup.}
Let  $\vartheta:\,E\to\bR$ be a continuously differentiable  convex function which is $1$-strongly convex with respect to the norm $\norm{\cdot}$, i.e.,
\[
\la\nabla \vartheta(x)-\nabla \vartheta(x'),x-x'\rangle\ge \|x-x'\|^2,\quad \forall x,x'\in E;
\]
such that $0 = \vartheta(0) \leq \vartheta(\cdot)\leq \Theta\|\cdot \|^2$ on $E$. $\Theta$ is the constant of quadratic growth of $\vartheta(\cdot)$; clearly, $\Theta \geq \half$. Set $\vartheta_{x_0}(\cdot) \eqdef \vartheta(\cdot-x_0)$ for some $x_0 \in \Xset$; it satisfies
\be
0 = \vartheta_{x_0}(x_0) \leq \vartheta_{x_0}(x)\leq \Theta\|x - x_0\|^2, \quad \forall x\in E.
\ee{qg1}
 Now, consider the {\em Bregman divergence} associated to the {\em distance-generating function (d.-g.f.)} $\vartheta_{x_0}$:
	\[
	V_{x_0}(x,z) \eqdef \vartheta_{x_0}(z)-\vartheta_{x_0}(x)-\la\nabla \vartheta_{x_0}(x),z-x\ra,
	\quad\forall\, z, x \in{\Xset}
	\]
and associated {\em proximal mapping} defined for $x\in \Xset$ and $u\in E$ according to
	\be
	\Prox_{x_0}(u,x) \eqdef  \argmin_{z\in \Xset}\big\{\inner{  u}{z} + V_{x_0}(x,z)\big\} = \argmin_{z\in \Xset}\big\{\inner{  u- \nabla \vartheta_{x_0}(x)}{z}+\vartheta_{x_0}(z)\big\}.
	\ee{prox}
When the norm $\|\cdot\|$ used is the Euclidean norm $\|\cdot\|_2$, we refer to the corresponding setup as Euclidean.
A natural choice of the d.-g.f. in this situation is $\vartheta(x) \eqdef \half \|x\|_2^2$. We refer to the setup as non-Euclidean otherwise. When  $\Theta$ is ``not too large,'' a high accuracy solution to the minimization problem $\min_{z\in \Xset} \{\inner{ u}{z}+V_{x_0}(x,z) \}$
	can be easily computed for all $x \in \Xset$, $u \in E$, following \cite{juditsky2011optimization,nemirovski2009robust,nesterov2013first} we say that  $\vartheta$ is ``prox-friendly.'' We present choices of prox-friendly d.-g.f.'s relative to the norm used in application sections.

\paragraph{SMD algorithm.}
	Given a positive stepsize sequence $\suite{\gamma_t}{t \geq 1}$,
	consider the SMD recursion, cf.  \cite{juditsky2011optimization,lan2012optimal,nemirovski2009robust},
	\be	x_{t+1}& \eqdef &\Prox_{x_0}\left(\gamma_{t+1} G(x_{t},\omega_{t+1}),x_{t}\right).
	\ee{prox1}
The approximate solution $\ov x_T$ after $T\geq 1$ iterations of the method is computed according to
\be\ov x_T \eqdef \frac1T \sum_{t=0}^{T-1} x_t.
\ee{xbar}
 Recall, that in the Euclidean setup, the recursion \rf{prox1} becomes the ``vanilla'' projected Stochastic Gradient Descent algorithm
 \[x_{t+1} = \Proj_{\Xset}\left( x_t - \gamma_{t+1} G(x_t, \omega_{t+1}) \right),
  \]
with the approximate solution $\ov x_T$ computed as the average of the iterates (here $\Proj_\Xset(\cdot)$ stands for the Euclidean projection on $\Xset$).

The next results describe some basic properties of the recursion \rf{prox1}.
We start with an analogue of the result we aim at in the case of independent $\omega_t$'s, when
Assumptions~\ref{hyp:problem}--\ref{hyp:oracle:Gt} define the standard framework of the study of the SMD algorithm.

Let us suppose that
$\suite{\omega_{t}}{t \geq 1}$ are independent random variables, and independent of $x_0$. In the present case, we consider the filtration $\F_0 \eqdef \sigma(x_0)$ and $\F_t \eqdef \sigma(x_0, \omega_1, \cdots, \omega_t)$; set $G_t(x) \eqdef g(x)$ and write  $\bE_t[G(x,\omega_{t+1})] =\bE[G(x,\omega_{t+1})]  = g(x)+\delta_t(x)$, implying that
\[
G(x,\omega_{t+1})=\xi_t(x,\omega_{t+1}) + g(x)+ \delta_t(x)
\]
with
\[ \xi_t(x,\omega_{t+1}) \eqdef G(x,\omega_{t+1})-\bE[G(x,\omega_{t+1})], \quad
\delta_t(x) \eqdef \bE[G(x,\omega_{t+1})]-g(x).
\]
By  \ref{hyp:oracle:decomp:delta},  $ \sup_{x \in \Xset} \sup_{t\geq 0}\normd{\delta_t(x)}\leq \delta$ and by \ref{hyp:oracle:decomp:xi},
\[ \bE_t\left[ \normd{\xi_t(x_t,\omega_{t+1})}^2\right] \leq \varsigma_*^2 + \kappa \normd{g(x_t)}^2;
\] \ref{hyp:oracle:Gt:lipschitz} and \ref{hyp:oracle:Gt:ell} become
\[\forall x,x' \in \Xset^0, \quad \normd{g(x) -g(x')} \leq \cL \norm{x-x'}, \quad \inner{g(x) - g(x')}{x-x_*} \leq \ell \normd{g(x)} \norm{x-x'},
\]
while \ref{hyp:oracle:Gt:mixing} trivially holds with $\sM_t = \sN_t = 0$ and any nonnegative $\suite{\rho_s}{ s \geq 0}$.
In the situation, we have the following result, cf. \cite[Proposition 2.1]{juditsky2020sparse}.
\begin{proposition}
    \label{prop:iid:case}
    Let Assumptions \ref{hyp:problem}-\ref{hyp:oracle:decomp} and \ref{hyp:oracle:Gt:lipschitz} hold.
Consider the SMD recursion \rf{prox1} with the constant stepsizes $\gamma_t \eqdef \gamma$ for all $t \geq 1$,  $\gamma \in(0, \left[\cL(1+2\kappa) \right]^{-1})$.
Let initial condition $x_0 \in\Xset$ be such that $\bE[\norm{x_0-x_*}^2]\leq r_0^2$. Then the approximate solution $\ov x_T$  (see \rf{xbar}) after $T$ iterations satisfies
\be
\bE\left[f(\ov{x}_T) \right] - f_* \lesssim \frac{\Theta r_0^2}{\gamma T}   +\varsigma_*^2 \gamma
+\delta\left(\sqrt{\Theta} r_0+\gamma\varsigma_* \sqrt{T}+\delta\gamma T\right).\ee{indc1}
Furthermore, assume that $T\ge 1$ and $\gamma>0$  obeys the bound
\be
\delta\lesssim \frac{\sqrt{\Theta}r_0}{\gamma T}+\frac{\varsigma_*}{\sqrt{T}}.
\ee{deltaindep}
Then
\be
\bE\left[f(\ov{x}_T) \right] - f_* \lesssim \frac{\Theta r_0^2}{\gamma T}   + \gamma \varsigma_*^2.\ee{ssb0}
In particular, for
\be
\gamma \eqdef \min\left\{\frac{1}{2\cL(1+2\kappa)},\sqrt{\frac{\Theta}T}\frac{r_0}{\varsigma_*}\right\},
\ee{gamma:indep}
we have
\be
\bE\left[f(\ov{x}_T) \right] - f_* \lesssim \frac{\Theta r_0^2 \cL(1+\kappa)}{T}   +\sqrt{\frac{\Theta}T}  r_0 \varsigma_*.
\ee{eq:indepbound}
\end{proposition}
For the sake of completeness, we provide the proof of Proposition \ref{prop:iid:case} in Appendix \ref{sec:proof:prop:iid:case}. Note that when $\varsigma_*=0$ and $\gamma = [2 \cL(1+2 \kappa)]^{-1}$, we get only the first term in the RHS of \eqref{eq:indepbound}.

In the unbiased case $\delta=0$, \eqref{ssb0} recovers, up to constants and differences in notation, the standard SMD bound; see, e.g., \cite[Proposition 2.1]{juditsky2020sparse}. More generally, the same bound continues to hold for a biased oracle provided that the bias is sufficiently small. Specifically, we refer to
\be
\delta\lesssim \max\left\{\frac{\sqrt{\Theta} r_0 \cL(1+\kappa)}{ T},\,\frac{\varsigma_*}{\sqrt{T}}\right\},
\ee{biasv1}
as the \emph{small bias condition}. Under this condition, the additional error caused by the bias is dominated by the baseline optimization and stochastic errors, so that neither the stepsize choice nor the resulting convergence rate is affected.
To streamline the presentation, we restrict attention to the setting relevant to our target applications, in which the small bias condition \rf{biasv1} is satisfied. We return to this condition and verify it explicitly in the corresponding application section.

%
%Let us denote $r_0$ the available upper bound for the distance to the minimizer of the initial condition $x_0\in\Xset$,
% $r_0\geq \bE[\norm{x_0-x_*}^2]^{1/2}$.
%
%Given the problem data $\Theta, \varsigma_*, \delta,$ and $ r_0$, we assume that $T\ge 1$ and the constant stepsizes $\gamma_t=\gamma>0$, $t=1,...,T$,  satisfy the relationship
%\be
%\delta\lesssim \max\left[\frac{\Theta^{1/2}r_0}{\gamma T},\,\frac{\varsigma_*}{\sqrt{T}}\right].
%\ee{biasv1}
%A straightforward analysis of the general bound of Theorem \ref{theo:a-sum1} of Appendix \ref{app:proof:SMDproof} reveals that relation \rf{biasv1} means
%
%
%We refer to this situation as that of {\em small bias.}

\begin{proposition}
\label{pro:sum1}
 Let Assumptions \ref{hyp:problem}--\ref{hyp:oracle:Gt} hold.
 Consider the SMD recursion \rf{prox1} with constant stepsizes
 $\gamma_t \eqdef \gamma$ for all $t \geq 1$, with $\gamma\in\left(0,[2\cL(1+2\kappa)]^{-1}\right]$.  Assume that $T\geq 2$, the initial condition $x_0 \in\Xset$ is such that $\bE[\norm{x_0-x_*}^2]\leq r_0^2$, that
the small bias condition \rf{biasv1} is satisfied, and $s\leq T/2$. Then for the approximate solution $\ov x_{T-s}$  it holds
\begin{align}
    \bE&\left[f(\ov{x}_{T-s})\right] -f_* \leq \frac{1}{T-s} \sum_{t=0}^{T-s-1} \bE\left[f(x_t)\right]-f(x_*) \leq
    \frac{1}{T-s} \sum_{t=0}^{T-s-1} \bE\left[\la g(x_t),x_t-x_*\ra\right]\nn
&\lesssim \left(\frac{\Theta r_0^2}{\gamma T} +\gamma \varsigma_*^2\right)
\big(1+\gamma^2s^2\ell^2\cL^2(1+\kappa)\big)+
 \gamma^2s^2\ell^2\cL\varsigma_*^2
 +\rho_s\left( \sqrt{\Theta} r_0   +\gamma \varsigma_*\sqrt{T}\right)\left( \ov M + D \ov N \right). \label{bp1sbias}
\end{align}
\end{proposition}
The proof of Proposition \ref{pro:sum1} is postponed to Appendix \ref{proof:sum1}.
\paragraph{Remarks.} A key difference, once the independence assumption on the noises $\omega_t$ is relaxed via the forgetting condition \ref{hyp:oracle:Gt:mixing}, is the introduction of the ``gap parameter'' $s$. Suppose that its value can be chosen so that the last term on the right-hand side of \rf{bp1sbias} becomes small relative to the sum of the other terms---in fact, up to an absolute constant, it does not exceed that sum.
Notice that this is certainly the case when
\be
\rho_s\leq(\ov M+D\ov N)^{-1}\left(\frac{\sqrt{\Theta} r_0}{\gamma T}+\frac{\varsigma_*}{\sqrt{T}}\right).
\ee{rho_ssm}
Indeed, when \rf{rho_ssm} holds, one has
\[
\rho_s\left(\sqrt{\Theta} r_0   +\gamma \varsigma_*\sqrt{T}\right) (\ov M+D\ov N)\leq \left(\frac{\sqrt{\Theta} r_0}{\gamma T}+\frac{\varsigma_*}{\sqrt{T}}\right)\left(\sqrt{\Theta} r_0   +\gamma \varsigma_*\sqrt{T}\right)\lesssim\frac{\Theta r_0^2}{\gamma T} + \gamma \varsigma_*^2.
\]
%By Proposition \ref{prop:iid:case}, the right-hand side above, when choosing $\gamma$ appropriately,  does not exceed (up to an absolute constant)
%\[\nu(r_0,T):= \frac{\Theta \cL(1+\kappa)r_0^2}{T}   +\varsigma_* r_0\sqrt{\frac{\Theta}T}.\]
%Let now $\ovs=\ovs(r_0,T)$ be the smallest integer $s$ such that
%\be
%\rho_s\leq(\ov M+D\ov N)^{-1}\nu(r_0,T).
%\ee{smalrho}
Furthermore, condition
\be
\rho_s\lesssim(\ov M+D\ov N)^{-1}\left(\frac{ \sqrt{\Theta} r_0 \cL (1+\kappa)}{ T}+\frac{\varsigma_*}{\sqrt{T}}\right),
\ee{rho_sssm}
guarantees bound \rf{rho_ssm} is admissible for every $\gamma\in\left(0,[2\cL(1+2\kappa)]^{-1}\right]$ under the premises of Proposition~\ref{pro:sum1}.

Let now $\ovs \eqdef \ovs(T,r_0;\cL,\Theta,\kappa,\varsigma_*)$ be the smallest integer $s$ such that the bound in \rf{rho_sssm} holds.
Assume that $\ovs\leq T/2$. Then
\begin{align}
    \bE\left[f(\ov{x}_{T-\ovs})\right] -f_* &\lesssim
    \left(\frac{\Theta r_0^2}{\gamma T} +\varsigma_*^2\gamma\right)
\big(1+\gamma^2\ovs^2\ell^2\cL^2(1+\kappa)\big)+
 \gamma^2\ovs^2\ell^2\cL\varsigma_*^2.\label{norhos}
 \end{align}
We can now specify the strategy for choosing the stepsize parameter of the method.

\begin{proposition}\label{pro:bgamm1}
Suppose that assumptions of Proposition \ref{pro:sum1} hold true.  Consider the SMD recursion \rf{prox1} with constant stepsizes $\gamma_t \eqdef \gamma$ for all $t\geq 1$, and assume that $\ovs \leq T/2$ is such that \rf{rho_sssm} is satisfied with $\rho_s \eqdef \rho_\ovs$.
\item[(i)] Suppose that $\ovs\ell\leq 2 (1+2\kappa)/\sqrt{1+\kappa}$ holds, and choose
\[\gamma \eqdef \min\left\{\frac{1}{2\cL(1+2\kappa)},\sqrt{\frac{\Theta}T} \frac{r_0}{\varsigma_*}\right\},
\]
then
\[
\bE\left[f(\ov{x}_{T-\ovs}) \right] - f_* \lesssim \frac{\Theta r_0^2\cL(1+\kappa)}{T}   +\sqrt{\frac{\Theta}T} r_0\varsigma_*.\]
\item[(ii)] Suppose that $\ovs\ell \geq  2(1+2\kappa)/\sqrt{1+\kappa}$,
    and choose
$\gamma \eqdef \min\{\gamma_1,\gamma_2,\gamma_3\}$ with
\[\gamma_1 \eqdef \frac1{\ovs\ell \cL \sqrt{1+\kappa})},\qquad
\gamma_2 \eqdef \sqrt{\frac{\Theta}T}\frac{r_0}{\varsigma_*},\qquad
\gamma_3 \eqdef  \left(\frac{\Theta r_0^2}{T\ovs^2 \ell^2  \cL \varsigma_*^2}\right)^{1/3},
\]
then
\be
\bE\left[f(\ov{x}_{T-\ovs}) \right] - f_* \lesssim \frac{\Theta r_0^2}T \ovs \ell \cL \sqrt{1+\kappa}+\sqrt{\frac{\Theta}T}\varsigma_*r_0+ \left(\frac{\Theta r_0^2 \ovs \ell\varsigma_*}T\right)^{2/3} \cL^{1/3}.
\ee{pro33b}
\end{proposition}
The proof of the proposition is postponed to Appendix \ref{sec:proof:pro:bgamm1}. In Statement {\em (ii)}, when $\varsigma_*=0$, $\gamma_2 = \gamma_3 = +\infty$ and hence $\gamma = \gamma_1$; \eqref{pro33b} becomes $ \bE\left[f(\ov{x}_{T-\ovs}) \right] - f_* \lesssim  \Theta r_0^2 \ovs \ell \cL \sqrt{1+\kappa}  / T$.

Statement {\em (i)} addresses the case of very fast ``forgetting,'' where $\rho_s$ vanishes fast as $s$ grows (e.g., the case of finite dependence, in which there is a small $\ov s>0$ such that $\omega_{t}$ is independent of $\omega_{t-s}$ for $s> \ov s$, for all $t\geq 1$). In this case, the choice of the constant stepsizes and the resulting SMD convergence rate are essentially the same as in the independent case (see Proposition~\ref{prop:iid:case}).  It is therefore statement {\em (ii)} of the proposition that is of interest in the case of dependent noises. Let us assume that $\rho_s$ decays geometrically, e.g.,
\[
\rho_s =\exp(-s/s_*),\quad s> 0,
\]
for some $s_*>0$. In this case, \rf{rho_sssm} holds for all
\[s\geq s_* \nu(T), \qquad  \nu(T) \eqdef \ln\left(\frac{(\ov M+D\ov N)T}{ \sqrt{\Theta} r_0 \cL(1+\kappa)+\varsigma_*\sqrt{T}}\right),
\]
and, as a consequence of Proposition \ref{pro:bgamm1}, we have:

\begin{corollary}\label{cor:rhoexp}
Suppose that $\rho_s = \exp(-s/s_*)$ for some $s_*> 0$, and let $\ovs \eqdef s_*\nu(T)$. Then, under the premises of Proposition \ref{pro:bgamm1}, the approximate solution $\ov x_{T-\ovs}$ by SMD with constant stepsize
\[
 \gamma \eqdef \min\left\{\frac{1}{s_*\nu(T)\ell \cL \sqrt{1+\kappa}},\,\sqrt{\frac{\Theta}T}\frac{r_0}{\varsigma_*},\,
\left(\frac{\Theta r_0^2}{T (s_* \nu(T))^2 \ell^2 \cL \varsigma_*^2}\right)^{1/3} \right\},
\]
satisfies
\be
\bE\left[f(\ov{x}_{T-\ovs}) \right] - f_* \lesssim \frac{\Theta\cL r_0^2}T  s_*\nu(T) \ell \sqrt{1+\kappa}+\sqrt{\frac{\Theta}T} r_0 \varsigma_*+\cL^{1/3} \left(\frac{\Theta r_0^2 s_*\nu(T) \ell\varsigma_*}T\right)^{2/3}.
\ee{eq:geom:galcase}
\end{corollary}
When $\varsigma_*=0$,  the statements remain valid with the convention that $1/0  =+ \infty$: the step size is $\gamma = [s_*\nu(T)\ell \cL \sqrt{1+\kappa}]^{-1}$ and the bound in \eqref{eq:geom:galcase} reduces to the first term.

\paragraph{Stochastic $\upsilon$-complexity.} We say {\em $\wh x\in\Xset$ is an approximate solution to \rf{prob1} of expected accuracy $\upsilon$} if
\[
\bE[f(\wh x)]-f_*\leq \upsilon.
\]
We can reformulate the results of Proposition \ref{pro:bgamm1} and Corollary \ref{cor:rhoexp} in terms of ``stochastic $\upsilon$-complexity" of the SMD---number $T=T(\upsilon)$ of successive calls to stochastic oracle (number of stochastic gradient observations) sufficient to attain ``the accuracy" $\upsilon$ of the approximate solution $\ov x_{t - \ovs}$. We present it here for the ``interesting case" corresponding to the statement $(ii)$ of Proposition \ref{pro:bgamm1}.
\begin{proposition}\label{cor:complex1}
Let Assumptions \ref{hyp:problem}--\ref{hyp:oracle:Gt} hold. Let  an initial condition $x_0 \in\Xset$  such that $\bE[\norm{x_0-x_*}^2]\leq r_0^2$. Moreover, given $\upsilon \in (0, \Theta r_0^2 \cL/2]$ and small absolute constants $c>0$ and $c'>0$,   set
\be
\ov s(\upsilon) \eqdef \min\left\{s\geq 0: \rho_s\leq \frac{c \, \upsilon}{(\ov M+D\ov N)\sqrt{\Theta} r_0}\right\}
\ee{ovsu}
and
\be\delta\leq c' \frac{\upsilon}{\sqrt{\Theta} r_0}.
\ee{deltau}
Assume that $\ovs(\upsilon)\ell> \frac{2(1+2\kappa)}{\sqrt{1+\kappa}}$.
Then the stochastic $\upsilon$-complexity  of the SMD recursion \eqref{prox1} with the constant stepsize
\[
\gamma(\upsilon) \eqdef \min\left\{\frac1{\ovs(\upsilon)\ell \cL \sqrt{1+\kappa}},\; \frac{\sqrt{\upsilon}}{\ov s(\upsilon) \ell \sqrt{\cL} \varsigma_*},\;  \frac{\upsilon}{\varsigma_*^2} \right\}
\]
% \aic{$T^*(\upsilon) \eqdef \Theta r_0^2 /(\upsilon \gamma(\upsilon))$, up to an absolute constant.
% By definition of $\gamma(\upsilon)$, we have
% }{}
does not exceed, up to an absolute constant, $T^*(\upsilon)$,
\be
T^*(\upsilon) = \max \left\{
\frac{\Theta r_0^2\ov s(\upsilon)\ell\cL \sqrt{1+\kappa}}\upsilon,\,\frac{\Theta r_0^2\ov s(\upsilon)\ell\sqrt{\cL}\varsigma_*}{\upsilon^{3/2}},\,
\frac{\Theta r_0^2 \varsigma_*^2}{\upsilon^2}
\right\}.
\ee{T*}
 \end{proposition}
The proof of the proposition is postponed till Appendix \ref{sec:proof:cor:complex1}.

\paragraph{Remarks.} Observe that for  $\upsilon \geq \Theta r_0^2 \cL$ already the initial $x_0$ is an $\upsilon$-solution to the problem.

Conditions \rf{deltau} and \rf{ovsu} are analogues of the small bias
condition \rf{biasv1} and the forgetting condition \rf{rho_sssm},
respectively. While the latter conditions ensure that the bias and
dependence errors are dominated by the baseline SMD error bound
\rf{ssb0}, conditions \rf{deltau} and \rf{ovsu} control these errors
directly relative to the target accuracy $\upsilon$.

When $\rho_s$ decays geometrically,
$\rho_s=\exp(-s/s_*)$ for some  $s_*>0$,  then $\ovs(\upsilon)$ is logarithmic in $\upsilon$,
 \[
 \ovs (\upsilon)  =  s_* \ln\left(
 \frac{(\ov M+D\ov N) \sqrt{\Theta} r_0}{ c \, \upsilon}\right),
 \]
and  the stochastic complexity of the SMD is bounded, up to an absolute constant, with the corresponding $T^*(\upsilon)$ in \rf{T*}.

Finally, in the case of $\varsigma_*=0$, the bound for the $\upsilon$-complexity of the SMD with the constant stepsize $\gamma(\upsilon) = (\ov s(\upsilon) \ell \cL \sqrt{1+\kappa})^{-1}$ reads $T^*(\upsilon) = \Theta r_0^2\ov s(\upsilon)\ell\cL \sqrt{1+\kappa}\, \upsilon^{-1}$,
up to an absolute constant.

\subsection{Multistage SMD algorithm for sparse optimization} \label{sec:multis}
% We assume to be given $R<\infty$ and $x_0\in \Xset$ such that $\norm{x_0 - x_*}\leq R$, along with problem parameters $\cL$, $\ell$, $\kappa$, and $\varsigma_*$.

Let $x_0\in\Xset$ be an initial point satisfying
$\norm{x_0-x_*}\leq R$ for some $R<\infty$,
and assume that the problem parameters
$\cL$, $\ell$, $\kappa$, and $\varsigma_*$ are known.
We suppose that stochastic optimization problem \rf{prob1} admits a sparse solution.
{Standard examples of the sparsity assumptions are as follows:}
% \footnote{For the sake of simplicity, we consider here two standard sparsity structures on $\bbr^n$. The case of another ``classical'' matrix sparsity---low rank structure may be addressed in the same way, see \cite{juditsky2014unified,juditsky2023sparse}.}
\begin{itemize}
    \item ``Standard'' sparsity: we assume that an optimal solution $x_*\in \Xset \subseteq \bR^n$ has at most $\chi\ll n$ nonvanishing entries. In this case, we put $\norm{\cdot}=\norm{\cdot}_1$ and $\normd{\cdot}=\norm{\cdot}_\infty$.
    \item Group sparsity: we suppose that the set $I$ of the indices of the entries of $x\in E$ is partitioned into $n$ subsets $\{I_1, ..., I_n\}$, and denote $x^j$ the $j$th block of $x$, meaning that $x^j \in E$ and $[x^j]_i=0$ for all $i\notin I_j$.  We assume that the optimal $x_*\in \Xset $ is a block vector with at most $\chi\leq n$ nonvanishing blocks $x_j$. We define
        $\norm{x} \eqdef \sum_{j=1}^n \norm{x^j}_2$ (block $\ell_1/\ell_2$-norm) and $\normd{x} \eqdef \max_{1 \leq j\leq n} \norm{x^j}_2$ (block $\ell_\infty/\ell_2$-norm).
    \item Low rank sparsity: consider the matrix space $\bR^{p\times q}$ where $p\geq q$ equipped with Frobenius inner product. We assume that the optimal solution $x_*\in \Xset$ satisfies $\rank(x_*)\leq \chi$. We consider the nuclear norm  $\norm{x} \eqdef \sum_{i=1}^q\sigma_i(x)$ where $\suite{\sigma_i(\cdot)}{1 \leq i \leq q}$ are the singular values of $x$, so that $\normd{x} \eqdef  \max_{1 \leq i \leq q}\sigma_i(x)$ is the spectral norm.
\end{itemize}

In what follows we assume
\begin{assumptionS}\label{hyp:sparse1}
%\item\label{hyp:sparse1}
The optimal solution $x_*\in\Xset$ to problem \rf{prob1} is $\chi$-sparse. An upper bound $\bar \chi$ for parameter sparsity $\chi$ is a priori known.
Furthermore, given $x\in \Xset$, one can efficiently compute a ``sparse approximation'' of $x$---an optimal solution $\sparse_\chi(x)$ to the optimization problem
\be
\min \norm{x-z}_2\;\;\mbox{over $\chi$-sparse $z\in \Xset$}.
\ee{sparseapp}
Moreover,  for any $\chi$-sparse $z\in E$ the norm $\norm{\cdot}$ satisfies $\norm{z}\leq \sqrt{\chi}\norm{z}_2$.
\end{assumptionS}

\begin{assumptionS}\label{hyp:sparse2}
%\item\label{hyp:sparse2:1}
Problem \rf{prob1} verifies a quadratic growth condition with respect to the Euclidean norm: there is $\underline \varkappa>0$ such that, for all  $x\in\Xset$,
\begin{align}\label{eq:quadractic_growth}
    f(x) - f_* \geq \tfrac{1}{2}\underline \varkappa\norm{x-x_*}_2^2.
\end{align}
\end{assumptionS}

\begin{assumptionS}\label{hyp:sparse3} The stochastic gradient observation $G(x_t,\omega_{t+1})$ available at the iterate $x_t$ of the stochastic algorithm is unbiased; i.e.,
    \[
    \bE_t[G(x_t,\omega_{t+1})]=G_t(x_t), \quad t=0, 1,...
    \]
\end{assumptionS}

\noindent Following \cite{ilandarideva2023stochastic,juditsky2023sparse}, we use the SMD algorithm and  apply the multistage modification of \cite{OML2011,juditsky2014deterministic} to improve its accuracy bounds.
The proposed SMD-SR works in stages, running the SMD algorithm at each stage with the subsequent ``sparsification'' of the approximate solution delivered by the SMD.

For $t=0,1,\dots$,
	we define the  SMD recursion with constant stepsize, cf.  \rf{prox1}, according to
	\be	x_{t+1}&=&\Prox_{x_0}(\gamma   G(x_{t},\omega_{t+1}),x_{t}).
	\ee{eq:md1s}
The approximate solution $\ov x_T$ to problem \rf{prob1} after $T$ iterations is defined as in \rf{xbar}.

\paragraph{Algorithm SMD-SR}
\begin{description}\item[]{\em Initialization:} Choose $\ov y_{-1}\in \Xset$ such that $\norm{\ov y_{-1} - x_*} \leq R$.
\item []{\em Stage k=0,...,K:}
\begin{enumerate}\item
Set
\[\upsilon_k \eqdef 2^{-k} \frac{\underline\varkappa}{32\bar\chi} R^2,\qquad r_0 \eqdef 2^{-k/2} R, \qquad x_0 \eqdef \ov y_{k-1}.
\]Given $r_0$ and the set $\cL,\, \ell,\,
\kappa,\, \varsigma_*,\, \ov N,\, \ov M$, and $\rho_s$, $s\in \bZ_{+}$ of problem parameters, for $T=1,2,...$, optimize w.r.t. $\gamma\in \left(0,[2\cL(1+2\kappa)]^{-1}\right]$ and $s\in \bZ_+$ the bound $\mr_T(\gamma,s)$ in \rf{mth1} of Theorem \ref{theo:a-sum1}, thus getting $\gamma(T)$ and $s(T)$. \\
Select $T=T_k$---the smallest integer such that\footnote{Note that this problem can be solved ``up to a moderate absolute constant'' by implementing a search over a geometric grid of values of $\gamma$, $s$, and $T$.}
\be
\mr_{T_k}(\gamma(T_k),s(T_k))\leq \upsilon_k
\ee{muk}

Run $T_k$ iterations of the SMD with the constant stepsize $\gamma=\gamma(T_k)$ and compute the approximate solution $\wh{x}_k(T_k) \eqdef \ov x_{T_k-s_k}$.
\item
Compute  $\ov y_k \eqdef \sparse_{\bar\chi}(\wh{x}_k)$.
\end{enumerate}
\item[]{\em Output:} the parameter estimate by SMD-SR after $K$ stages $\wh x \eqdef \ov y_K$.
\end{description}

\begin{proposition}\label{prop:SR}
Suppose that Assumptions \ref{hyp:problem}--\ref{hyp:oracle:Gt} and \ref{hyp:sparse1}--\ref{hyp:sparse3} hold true.
\item[(i)] Let $\suite{(\wh x_k, \ov y_k)}{0 \leq k \leq K}$ be computed by the SMD-SR algorithm with initial condition $x_0$ satisfying $\norm{x_0-x_*}\leq R$. Then for all $k\geq 0$
	\[
		\bE[f(\wh x_k)] - f_* \lesssim 2^{-k}{ R^2}\ul\varkappa/{\chi} , \quad\bE[\norm{\ov y_k-x_*}^2_2]\lesssim  2^{-k}{R^2}/{\ov \chi}, \quad \bE[\norm{\ov y_k-x_*}^2]\lesssim 2^{-k}R^2,
	\]
implying, in particular, that, given a precision $\ov\upsilon>0$, for\footnote{Here $\lceil a\rceil$ is the upper integer part---the smallest integer greater or equal to $a$.} $K\geq K(\ov\upsilon) \eqdef \big\lceil2\log_2 \left[{R}/{\ov\upsilon}\right]\big\rceil$ one has
\be
\bE\left[\norm{\wh x-x_*}^2\right]\leq \ov\upsilon^2.
\ee{snorm}
For $k\geq 0$, denote
\be
\ovs_k \eqdef \min\left\{s\geq 0: \rho_s\leq  2^{-k/2} c [(\ov M+D\ov N) \sqrt{\Theta} \bar \chi]^{-1} R \ul\varkappa \right\},
\ee{ovsk}
where $c$ is a small absolute constant; we assume that $\ovs_k \ell \geq  \frac{2(1+2\kappa)}{\sqrt{1+\kappa}}$.
Let
\be
T^*_k \eqdef \max\left\{\frac{\Theta\ovs_k\ell\cL
\sqrt{1+\kappa}\bar\chi}{\ul{\varkappa}},\,2^{k/2}\frac{\Theta\ovs_k\ell\sqrt{\cL}\varsigma_*\bar\chi^{3/2}}{\ul\varkappa^{3/2} R},\;
2^k\frac{\Theta\varsigma_*^2\bar\chi^2}{\ul\varkappa^2 R^2}
\right\}.
\ee{tkstage}
The  duration $T_k$ (the iteration count) of the $k$th stage of the algorithm is bounded with $T^*_k$, up to an absolute constant.
\item{(ii)}
Let us put
\be
T^*(\ov \upsilon) \eqdef \max\left\{\frac{\Theta\ell\cL \sqrt{1+\kappa}\bar\chi}{\ul{\varkappa}}\sum_{k=0}^{ K_*(\ov\upsilon)}\ovs_k,\;
\frac{\Theta\ell\sqrt{\cL}\varsigma_*\bar\chi^{3/2}}{\ul\varkappa^{3/2} R}\sum_{k=0}^{ K_*(\ov\upsilon)}2^{k/2}\ovs_k,\;\frac{\Theta\varsigma_*^2\bar\chi^2}{\ul\varkappa^2 \ov \upsilon^2}\right\}.
\ee{Tov}
Then the $\ov\upsilon$-stochastic complexity of the SMD-SR algorithm is bounded by $T^*(\ov\upsilon)$, up to an absolute constant. In other words,
the approximate solution $\wh x$ by SMD-SR algorithm totalizing $\cT$ iterations ($\cT$ calls to stochastic oracle) satisfies
\[\bE\left[\norm{\wh x-x_*}^2\right]\leq \ov\upsilon^2\]
for $\cT\gtrsim T^*(\ov\upsilon)$.
 \end{proposition}
The proof of the proposition  is  postponed till Appendix~\ref{sec:proof:SMDSR}.
We have the following version of the above result when $\rho_s$ decays at a geometric rate.
\begin{corollary}\label{cor:SR}
Under the premises of Proposition \ref{prop:SR}, suppose that
$\rho_s=\exp(-s/s_*)$ for some $s_*>0$.
In this case, the $\ov\upsilon$-complexity of the SMD-SR is bounded, up to an absolute constant, with $\ov T^*(\ov\upsilon)$,
 \be
\ov T^*(\upsilon)\leq \max\left\{\frac{\Theta\ell\cL \sqrt{1+\kappa} \bar\chi s_*\nu^2(\upsilon)}{\ul{\varkappa}},\;
\frac{\Theta \ell \sqrt{\cL} \varsigma_*  \bar\chi^{3/2} s_*\nu(\upsilon)}{\ul\varkappa^{3/2} \upsilon},\;\frac{\Theta\varsigma_*^2\bar\chi^2}{\ul\varkappa^2 \upsilon^2}\right\}
\ee{Tove}
with $\nu(\upsilon) \eqdef \ln\left((\ov M+D\ov N)\sqrt{\Theta}\bar\chi/(\ul\varkappa\upsilon)\right)$.
\end{corollary}

The proof of the corollary is postponed till Appendix~\ref{sec:proof:cor:SR}.

\paragraph{Remarks.}
The stages of SMD-SR can be divided into three groups, or phases,
corresponding to three different convergence regimes. These regimes
are characterized by the three terms in the stage complexity bound \rf{tkstage} (see also the corresponding terms in the convergence rate bound \rf{pro33b}).

During the first, preliminary phase, the first term on the
right-hand side of \rf{tkstage} dominates. This corresponds to the dominance
of the first term in \rf{pro33b} for the expected suboptimality of
the approximate solution $\bar{x}_{T_k-s_k}$. During this phase,
the number of iterations per stage, or the stage length, is
proportional to $s_k$. Under geometric forgetting, $s_k$ grows
at most linearly with the stage index $k$. Since the expected
squared error of $y_k$ decreases by a constant factor at each
stage, the algorithm exhibits nearly linear convergence during
this phase, with slowly increasing stage lengths.

When the second term in \rf{tkstage} becomes dominant over the first,
the preliminary phase ends and the algorithm enters the
intermediate phase. In this regime, the stage length $T_k$ is
inversely proportional to $2^{-k/2}R$ if we ignore the dependence
of $s_k$ on $k$, while the expected squared error of $y_k$
decreases as $2^{-k}R^2$. Thus, in terms of the total number
$T$ of stochastic oracle calls, the expected squared error
decreases at essentially an $O(1/T^2)$ rate during this phase.

Finally, when the third term in \rf{tkstage} becomes dominant, the
algorithm enters the asymptotic phase. In this regime, the
stepsize choice and stage lengths are standard for multistage
SMD, and the expected squared error decreases at the standard
$O(1/T)$ rate. Notably, this rate does not depend on the
forgetting parameters $\rho_s$, $\bar M$, and $\bar N$.
However, these parameters affect when the asymptotic phase begins.

\section{Application: High-dimensional Linear Regression for Time-series Estimation}
\label{sec:appl}

\subsection{Linear regression by stochastic approximation} \label{sec:appl:linearreg}
Let us consider the problem of recovery of a parameter matrix $x_*\in\bR^{p\times n}$,  from  observations
\be
\upsilon_{t+1}=x_*\phi_t+\veps_{t+1},\quad t=0,1,2,...,T-1,
\ee{sparselin}
where $\suite{\phi_t}{t \geq 0}$ are  $\bR^n$-valued random regressors, and $\suite{\veps_t}{t \geq 1}$  are (unobserved) $\bR^p$-valued perturbations. In this section we assume that
\begin{assumptionR}\label{aF:01:decomp}%
$\suite{\veps_t}{t\geq 1}$ allow for the decomposition $\veps_t=\epsilon_{t-1}+\zeta_t$ where
$\suite{\zeta_t}{t \geq 1}$ are $\bR^p$-valued independent random vectors, independent of $\phi_0$. In addition,
$\bE[\zeta_t]=0$ and  there exists $\sigma_*>0$ such that  $\sup_{t \geq 1} \bE[\norm{\zeta_t}_2^2]\leq \sigma_*^2$.
\end{assumptionR}
We define the filtration $\F_t \eqdef \sigma(\phi_0,\zeta_0, \cdots,\zeta_t)$, for all $t\geq 0$; and  denote $\bE_t[\cdot]$ the conditional expectation w.r.t. $\F_t$.
\begin{assumptionR}\label{aF:01:phi}%
For all $t \geq 1$, $\phi_t$ is $\F_t$-measurable. Vectors $\phi_t$ are uniformly  bounded: there is a (deterministic) constant $\ov \phi$ such that $\sup_{t \geq 0} \norm{\phi_t}_\infty\leq \ov\phi$. Furthermore, there is a distribution $P_\phi$ on $\bR^n$ such that
\[
\bE_{\phi\sim\P_\phi}[\phi\phi^\top]=:\Sigma\succ 0
\]
(i.e., $\Sigma$ is positive definite), and there is a positive sequence $\suite{\rho_s}{s \geq 0}$ and $\F_t$-adapted random variables $\suite{\sN_t}{t\geq 0}$, such that $\sup_{t \geq 0} \bE[\sN_t^2]\leq \ov N^2<\infty$, and  for all $t,s\geq 0$,
\footnote{Recall that for a matrix $Q$ we denote $\|Q\|_\infty=\max_{ij}|[Q]_{ij}|$.}
\be
\left\|\bE_t\left[\phi_{t+s}\phi_{t+s}^\top\right]-\Sigma
\right\|_\infty\leq \rho_s \sN_t.
\ee{phiP}
\end{assumptionR}
\begin{assumptionR}
 \label{aF:01:x_*}%
We suppose that $x_*\in \Xset$ where $\Xset$ is a given a convex and compact subset of $\bR^{p\times n}$; the $\norm{\cdot}_{2,1}$-diameter of $\Xset$ is bounded with $D$, i.e., $\max_{x,x'\in \Xset}\norm{x-x'}_{2,1}\leq D$.  \aic{We also suppose that we are given an $x_0\in\bR^{p\times n}$ and $r_0>0$ such that $\bE[\norm{x_0-x_*}^2]\leq r_0^2$.}{}
\end{assumptionR}
\begin{assumptionR}
\label{aF:01:eps}%
For all $t \geq 1$, $\epsilon_t$ is $\F_t$-measurable. There is $\ov\epsilon>0$, such that $\sup_{t}\bE[\norm{\epsilon_t}_2^2]\leq \ov\epsilon^2.$  Moreover, one has
\[
\ov\epsilon\lesssim\frac{\ov\phi r_0\sqrt{\ln(1+n)}}T+\frac{\sigma_*}{\sqrt{T}}.
\]
\end{assumptionR}
To estimate the regression parameter $x_*$, we are about to apply Stochastic Optimization approach described in Section \ref{sec:probs}. To this end, consider
the optimization problem
\be
\min_{x\in \Xset}\left\{f(x) \eqdef \half\bE_{\phi\sim P_\phi}\big[\norm{(x-x_*)\phi}_2^2\big]=\half\Tr\left((x-x_*)\Sigma(x-x_*)^\top\right)\right\}.%\eqno{(SO)}
\ee{SOP}
$x_*$ is clearly the unique optimal solution to the above problem.\footnote{Indeed, when denoting $z_i\in\bR^n$ the columns of $(x-x_*)^\top$, one has, due to the positive definiteness of  $\Sigma$, $\Tr\left((x-x_*)\Sigma(x-x_*)^\top\right)=\sum_{i}z_i^\top\Sigma z_i>0$ if any of $z_i\neq 0$.}
Given an observation $ \omega_{t+1} \eqdef (\phi_t,\upsilon_{t+1})$ one can build a stochastic estimate $G(x,\omega_{t+1})$ of the gradient
$g(x) \eqdef (x-x_*)\Sigma$ of $f(x)$:
\be
G(x,\omega_{t+1}) \eqdef (x\phi_t-\upsilon_{t+1})\phi_t^\top=(x-x_*)\phi_t\phi_t^\top-\varepsilon_{t+1}\phi_t^\top.
\ee{sgor}
Denote $G_t(x)\eqdef (x-x_*)\phi_t\phi_t^\top$. Observe that by \ref{aF:01:decomp}
\[\xi_{t+1} \eqdef G(x,\omega_{t+1})-\bE_t\big[G(x,\omega_{t+1})\big]= - \zeta_{t+1}\phi_t^\top,
\]
so that by \ref{aF:01:decomp} and \ref{aF:01:phi},
\[\bE_t\left[\norm{\xi_{t+1}}_{2,\infty}^2\right]\leq \bE_t\big[\norm{\zeta_{t+1}}_2^2 \norm{\phi_t}_\infty^2 \big] \leq \sigma_*^2\ov\phi^2=:\varsigma_*^2. \]
Moreover, by \ref{aF:01:decomp},
\[
\bE_t\big[G(x,\omega_{t+1})\big]-G_t(x)= - \epsilon_t\phi_t^\top=:\delta_t
\]
and, by \ref{aF:01:phi} and  \ref{aF:01:eps},
\[\bE\left[\norm{\delta_t}_{2,\infty}^2\right]\leq \bE\left[\norm{\epsilon_t}_{2}^2 \norm{\phi_t}_\infty^2\right]\leq \ov\phi^2\ov\epsilon^2 =: \delta^2.
\]
Obviously, $G_t(\cdot)$ is a gradient of the convex quadratic $F_t(x) \eqdef \half \norm{(x-x_*)\phi_t}_2^2$, and $x \mapsto G_t(x)$ is Lipschitz-continuous with Lipschitz constant $\cL \eqdef \ov\phi^2$, and such that
\[
\la G_t(x)- G_t(x'),x-x_*\ra=\la (x-x')\phi_t\phi_t^\top,x-x_*\ra=\Tr\left((x-x')\phi_t\phi_t^\top(x-x_*)^\top\right)\leq
\norm{x-x'}_{2,1}\norm{(x-x_*) \phi_t\phi_t^\top}_{2,\infty}.
\]
Finally, for $s,t\geq 0$, due to \rf{phiP},  we have\footnote{Indeed, let $a\in\bR^{p\times n}$ and $b\in \bR^{n\times n}$ with columns $b_i$, $i=1,...,n$. Then,
$\norm{ab}_{2,\infty}=\max_{i} \norm{\Col_i[ab]}_2=\max_{i} \norm{ab_i}_2\leq \max_i\norm{a}_{2,1}\norm{b_i}_\infty=\norm{a}_{2,1}\norm{b}_\infty$.}
\begin{align*}
 \left\|\bE_t\left[G_{t+s}(x)-g(x)\right]\right\|_{2,\infty}&=\left\|\bE_t\left[(x-x_*)(\phi_{t+s}\phi_{t+s}^\top-\Sigma)\right]\right\|_{2,\infty}
\\& \leq \norm{x-x_*}_{2,1}\left\|\bE_t\left[\phi_{t+s}\phi_{t+s}^\top\right] - \Sigma\right\|_{\infty}
\leq \rho_s \sN_t  \norm{x-x_*}_{2,1}.
\end{align*}
We conclude that, when choosing the primal norm $\norm{\cdot}\eqdef \norm{\cdot}_{2,1}$ (with the conjugate norm $\normd{\cdot} \eqdef \norm{\cdot}_{2,\infty}$) and  setting $\varsigma_* \eqdef \sigma_*\ov\phi$, $\delta \eqdef \ov\phi\ov\epsilon$, $\cL \eqdef \ov\phi^2$, $\ell=1$, and $\kappa\eqdef 0$,
under \ref{aF:01:decomp}--\ref{aF:01:eps}, the observation $G(x, \omega_{t+1})$ of the gradient $g(x)$  satisfies assumptions \ref{hyp:problem}-\ref{hyp:oracle:Gt} in Section \ref{sec:probs} with $\ov M=0$ and $\ov N$ and $\rho_s$ as defined in \ref{aF:01:phi}.
\paragraph{$\ell_2/\ell_1$-proximal setup \cite[Section 2.1]{nesterov2013first}.}
We put
\[
q \eqdef \left\{\begin{array}{ll} 2,&n\leq 2,\\1+[\ln n]^{-1},&n\geq 3,\end{array}\right., \quad \lambda \eqdef \left\{
\begin{array}{ll} 1,&n=1,\\\half,&n=2,\\ {[e\ln n]^{-1}},&n\geq 3.\end{array}\right.
\]
For $x\in\bR^{p\times n}$ with columns $\Col_j[x],\,j=1,...,n$ the distance-generating function $\vartheta(x)$ is defined according to
\[
\vartheta(x) \eqdef \frac{n^{(q-1)(2-q)/q}}{2\lambda}\left(\sum_{j=1}^n\norm{\Col_j[x]}_2^q\right)^{2/q}.
\]
By \cite[Corollary 2.2]{nesterov2013first}, $\vartheta(\cdot)$ is strongly convex with respect to $\|\cdot\| \eqdef \|\cdot\|_{2,1}$ with unit convexity parameter, and satisfies
$\vartheta(\cdot)\leq \Theta\|\cdot\|^2$ with $\Theta = 2 e \ln(n+1)$.

Observe that, by \ref{aF:01:eps},
\be
\delta =\ov\phi\ov\epsilon\lesssim \frac{\ov\phi^2 r_0\sqrt{\ln(n+1)}}T+\frac{\ov\phi \sigma_*}{\sqrt{T}}\lesssim
\frac{\sqrt{\Theta}r_0 \cL}T+\frac{\varsigma_*}{\sqrt{T}}
\ee{deltareg}
satisfies the small bias condition of Proposition \ref{pro:bgamm1}, cf. \rf{biasv1}. We have the following corollary of Proposition  \ref{pro:bgamm1}.
\begin{proposition}\label{pro:bgamm1lr}
Suppose that assumptions \ref{aF:01:decomp}-\ref{aF:01:eps} hold true. Under the $\ell_2/\ell_1$-proximal setup, consider the SMD recursion \rf{prox1} with constant stepsizes $\gamma_t=\gamma>0$. Assume that the initial condition $x_0 \in\Xset$ is such that $\bE[\norm{x_0-x_*}^2]\leq r_0^2$  and that
 $\ovs\in [2,T/2]$ is such that \rf{rho_sssm} is satisfied with $\rho_s=\rho_\ovs$.
Let
$\gamma \eqdef \min\{\gamma_1,\gamma_2,\gamma_3\}$ with
\[\gamma_1 \eqdef \frac1{\ovs \ov\phi^2},\quad
\gamma_2 \eqdef \sqrt{\frac{\ln(n+1)}T}\frac{r_0}{\sigma_* \ov\phi},\quad
\gamma_3 \eqdef \left(\frac{\ln(n+1) r_0^2}{T \ovs^2  \sigma_*^2 \ov\phi^4}\right)^{1/3}.
\]
Then
\[
\bE\left[f(\ov{x}_{T-\ovs}) \right] - f_* \lesssim
\frac{\ln(n+1) r_0^2}{T}\ovs \ov\phi^2+\sqrt{\frac{\ln(n+1)}T} r_0 \sigma_* \ov\phi +\left(\frac{\ln(n+1) r_0^2 \ovs\sigma_* \ov\phi^2}T\right)^{2/3}.
\]
\end{proposition}

\paragraph{Column-sparse linear regression.}
In the situation of this subsection,
let us assume that the regression parameter $x_*\in \bR^{p\times n}$ possesses a block sparsity structure. Specifically, we assume that
\begin{assumptionSR}\label{hyp:sparse211}
%\item \label{hyp:sparse211:1}
$x_*\in\Xset$ is column-sparse---at most $\chi\ll n$ columns of $x_*$ do not vanish. An upper bound $\bar\chi$ for parameter sparsity $\chi$ is known.
\end{assumptionSR}
\begin{assumptionSR}\label{hyp:sparse211:2}
There is $\underline \varkappa>0$ such that
\be\Sigma \eqdef \bE_{\phi\sim\P_\phi}[\phi\phi^T]\succeq \underline \varkappa I
\ee{sigmavar}
(matrix $\Sigma$ is positive-definite with the smallest eigenvalue $\geq \underline \varkappa$).
\end{assumptionSR}
\begin{assumptionSR}\label{hyp:sparse211:3}
We suppose that the ``bias component'' $\epsilon_{t}$ of the observation error $\varepsilon_{t+1}$ vanishes: $\epsilon_t=0$ for $t=0,1,...$.
\end{assumptionSR}
To recover the model parameter $x_*$ we use SMD-SR algorithm under $\ell_2/\ell_1$-proximal setup as applied to the optimization problem \rf{SOP}~\footnote{It is a special instance of the group sparsity case (see Section~\ref{sec:multis}) with the corresponding sets $I_j$ collecting the indices $\cdot j$ of the elements of the $j$th column of $x$.}.
Note that, in the present situation, when assuming, for instance, that the set  $\Zset \subset \bR^n_+$ of vectors $z \eqdef \left[\|\Col_1[x]\|_2;...;\|\Col_n[x]\|_2\right]$ is monotone, i.e., whenever $0\leq z'\leq z\in\Zset$ one has $z'\in\Zset$, computing $x_{\bar \chi} \eqdef \sparse_{\bar \chi}(x)$ for $x\in \Xset$ is straightforward by setting to zero all columns of $x$ except for $\bar \chi$ columns of largest $\ell_2$-norms. We use observations \rf{sparselin} to build stochastic gradients  $G(x_t,\omega_{t+1})$, cf. \rf{sgor}.

Clearly, Assumptions \ref{hyp:sparse211}--\ref{hyp:sparse211:3} imply the corresponding sparsity assumptions of Proposition \ref{prop:SR}.  In particular,  \ref{hyp:sparse211:2} implies \ref{hyp:sparse2} with the quadratic growth condition for the objective  $f(\cdot)$  in \rf{SOP}.
% \[
% f(x) - f_* \geq \tfrac{1}{2}\underline \varkappa\norm{x-x_*}_2^2.
% \]
When combining the derivations in this section with the result of Proposition~\ref{prop:SR}, we arrive at the following
\begin{corollary}\label{cor:sparser}
Suppose that Assumptions \ref{aF:01:decomp}--\ref{aF:01:eps} and \ref{hyp:sparse211}--\ref{hyp:sparse211:3} hold true. Under the $\ell_2/\ell_1$-proximal setup, consider the SMD-SR algorithm with the initial condition $\ov y_{-1} \in\Xset$ satisfying $\norm{\ov y_{-1} -x_\star}\leq R$.
Denote
\[
\ovs_k \eqdef \min\left\{s\geq 2: \rho_s\leq  2^{-k/2} c [D\ov N\sqrt{\ln(n+1)}\ov\chi]^{-1} R \ul\varkappa \right\}
\]
where  $c>0$ is a small absolute constant, and for $\upsilon>0$, $K_*(\upsilon) \eqdef \big\lceil2\log_2 \left[{R}/{\upsilon}\right]\big\rceil$, and
\[
T^*(\upsilon) \eqdef \max\left\{\frac{\ln(n+1)\ov\phi^2\bar\chi}{\ul{\varkappa}}\sum_{k=0}^{ K_*(\upsilon)}\ovs_k,\;
\frac{\ln(n+1)\ov\phi^2\sigma_*\bar\chi^{3/2}}{\ul\varkappa^{3/2} R}\sum_{k=0}^{ K_*(\upsilon)}2^{k/2}\ovs_k,\;\frac{\ln(n+1)\sigma_*^2\ov\phi^2\bar\chi^2}{\ul\varkappa^2 \upsilon^2}\right\}.
\]
The $\upsilon$-observation complexity of the SMD-SR algorithm is bounded by $T^*(\upsilon)$, up to an absolute constant. In other words,
the approximate solution  $\wh x$ by SMD-SR algorithm utilizing $\cT$ observations \rf{sparselin} satisfies
\[\bE\left[\|\wh x-x_*\|^2\right]\leq \upsilon^2\]
for $\cT\gtrsim T^*(\upsilon)$.
\end{corollary}

\subsection{Estimation of functional autoregressive model} \label{sec:example:VFAR}
%%%%%%%%%%%%%%%%%%%%%%%%%%%%%%%%%%%%%%%%%%%%%%%%%%%%%%%%%
We consider the problem of estimating the function $r(\cdot)$ in the {\em nonparametric autoregressive model}
\be\eta_{t+1}=r(\eta_{t})+\zeta_{t+1},   \quad\eta_{t}\in\bR^p,
\ee{mc0}
where   $r:\,\bR^{p}\to\bR^p$ is a nonlinear mapping~\footnote{A closely related problem consists in estimating the ``drift function" $\ov r(\cdot)$ such that $r(u)=u-\ov r(u)$.}. It is assumed that
\begin{assumptionF} \label{vfar:zeta}
The random variables  $\suite{\zeta_t}{ t \geq 1}$ are i.i.d. and independent of $\eta_0$; $\bE[\zeta_1]=0$, and there exists  $\sigma_* >0$ and a positive-definite matrix $Q$ such that  $\bE[\norm{\zeta_1}_2^2]\leq \sigma_*^2$, and  $\bE\left[ \exp(\normA{\zeta_1}) \right]< \infty$.  Moreover, for some $M>0$,
\begin{equation}\label{eq:vfar:contraction}
\bE\left[ \exp( \normA{\zeta_1}) \right]  \sup_{ \normA{u} >M}     \exp(  \normA{r(u)} - \normA{u}) < 1.
\end{equation}
Finally, $\bE\left[ \exp( \normA{\eta_0})\right]<\infty$.
\end{assumptionF}
Note that \rf{eq:vfar:contraction} holds, for instance,  if there exist a compact set $\KK$,  $\mu \in \ooint{0,1}$, and $\mathsf{\alpha}>0$ such that for all $u \notin \KK$,
\[\normA{r(u)} \leq \normA{u} -\alpha \normA{u}^{\mu} .\]
Consider also the following assumptions.
\begin{compassumptionF} \label{vfar:r}
\item \label{vfar:r:item2} $\zeta_1$ possesses a density w.r.t. the Lebesgue measure on $\bR^p$ which is positive and lower semicontinuous on $\bR^p$; furthermore, $r$ is continuous on $\bR^p$.
\item  \label{vfar:r:item1} For all $u,u' \in \bR^p$, $\normA{r(u) - r(u')} \leq \normA{u - u'}$. In addition, for any compact set $\mathsf{K}$ of $\bR^p$, there exists $\epsilon_{\mathsf{K}} \in \ooint{0,1}$ such that for all $u, u' \in \mathsf{K}$: $\normA{ r(u) - r(u')} \leq (1- \epsilon_{\mathsf{K}}) \normA{ u - u'}$.
\end{compassumptionF}
Our objective is, given observations $\eta_0,...,\eta_T$ of the Markov chain \rf{mc0}, to build an estimate of the function $r(\cdot)$.
To this end, we approximate $r(\cdot)$ by linear combination of $n$ ``basis functions'' $\Phi_j(\cdot):\,\bR^p \to \bR$, $j=1,...,n$;
we denote \[\Phi(u)=[\Phi_1(u);...;\Phi_n(u)] \in \bR^{n \times 1}.\]
We will see in an instant that \ref{vfar:zeta} and \ref{vfar:r} imply (see Lemma~\ref{lem:vfar0})  that the Markov chain $\suite{\eta_t}{t\geq 0}$ has a unique invariant distribution which we denote $\pi$. This allows us to define the ``target approximation'' $r_*(u) \eqdef x_*\Phi(u)$ where $x_*$ is an optimal solution of the stochastic optimization problem
\be
\min_{x} \left\{f(x) \eqdef \half\bE_{\eta\sim\pi}\left[\norm{r(\eta)-x\Phi(\eta)}_2^2\right]\right\}
\ee{napmin2}
which is  equivalent, under \ref{vfar:zeta}, to
\[
\min_x\half\bE_{(\eta, \eta') \sim \pi(\rmd \eta) \P(\eta, \rmd \eta')}\left[\|\eta'-x\Phi(\eta)\|_2^2\right]%, \qquad (\eta, \eta') \sim \pi(\rmd \eta) \P(\eta, \rmd \eta'),
\]
where $\P$ is the transition kernel of the Markov chain $\suite{\eta_t}{t \geq 0}$: for $u\in \bR^p$ and any Borel set $\AA\subset \bR^p$,
    \[
    \P(u,\AA) \eqdef \Prob\left(r(u) +\zeta_1 \in \AA \right).% = \int_\AA \mathsf{p}_\zeta(z -r(u))  \rmd z.
    \]
\begin{assumptionF} \label{vfar:X}
We assume to be given a convex and compact subset $\Xset$ of $\bR^{p\times n}$ such that $x_*\in\Xset$, with the diameter, as measured in $\|\cdot\|_{2,1}$-norm, not exceeding a given $D<\infty$: $\sup_{x,x'\in \Xset}\norm{x-x'}_{2,1}\leq D$.
\end{assumptionF}
We make the following assumptions about the functional family $\suite{\Phi_j}{j=1,...,n}$.
\begin{compassumptionF} \label{vfar:phi}
\item \label{vfar:phi:item1} Functions $\Phi_i$ are uniformly bounded: one has $\norm{\Phi_i}_\infty\leq \ov \phi$ for some $\ov\phi<\infty$.
We also suppose that the matrix $\Sigma \eqdef \bE_{\eta \sim \pi}\left[ \Phi(\eta) \Phi(\eta)^\top \right]$ is positive-definite.
\item  \label{vfar:phi:item2}  Functions $\Phi_i$ are also Lipschitz-continuous:  $|\Phi_i(u)-\Phi_i(u')|\leq L\normA{u-u'}$, $i=1,...,n$.
\end{compassumptionF}
Observe that \ref{vfar:phi}-item~\ref{vfar:phi:item1} implies that
$
|\Phi_i(u) \Phi_j(u) - \Phi_i(u') \Phi_j(u')| \leq 2\ov\phi^2$,
and \ref{vfar:phi}-item
implies that
\[
|\Phi_i(u) \Phi_j(u) - \Phi_i(u') \Phi_j(u')| \leq |\Phi_i(u)(\Phi_j(u) - \Phi_j(u'))|+|\Phi_j(u')(\Phi_i(u)-\Phi_i(u')|\leq
2\ov\phi L\normA{u-u'}.
\]
\begin{assumptionF} \label{vfar:delta}
We suppose that approximation $r_*(\cdot) \eqdef x_*\Phi(\cdot)$ of $r(\cdot)$ satisfies
\[
\sup_t\bE\left[\|r_*(\eta_t)-r(\eta_t)\|_2^2\right]\leq \ov\epsilon^2
\]
with
\[
\ov\epsilon\lesssim\frac{\sqrt{\ln[1+n]} r_0 \ov\phi}T+\frac{\sigma_*}{\sqrt{T}}.
\]
\end{assumptionF}Given observations $\eta_0,...,\eta_T$, we put  $\phi_t \eqdef \Phi(\eta_t)$, $\epsilon_t \eqdef r(\eta_t)-x_*\phi_t$, and $\varepsilon_{t+1} \eqdef \epsilon_{t}+\zeta_{t+1}$, we rewrite \rf{mc0} thus arriving at
\be
\eta_{t+1}=x_*\phi_t+\varepsilon_{t+1},\quad t=0,...,T-1.
\ee{etatr}
Following the same lines as in Section~\ref{sec:appl:linearreg}, the optimization problem \eqref{napmin2} is solved by the SMD recursion
\[
x_{t+1} = \Prox_{x_0}\left(-\gamma_{t+1}  \left( \eta_{t+1} - x_t \Phi(\eta_t) \right) \Phi(\eta_t)^\top,x_{t}\right);
\]
here $G(x_t, \omega_{t+1}) \eqdef - \left( \eta_{t+1} - x_t \phi_t \right) \phi_t^\top$  where $\omega_{t+1} \eqdef (\phi_t, \eta_{t+1})$, is a stochastic oracle of the gradient $g(x) \eqdef  -\bE_{(\eta, \eta') \sim \pi(\rmd \eta) \P(\eta, \rmd \eta')}\left[\left(\eta'-x\Phi(\eta)\right) \Phi(\eta)^\top\right]$ at $x=x_t$.
Our next objective is to verify Assumptions~\ref{aF:01:decomp}--\ref{aF:01:eps} of Section~\ref{sec:appl:linearreg} for our observation scheme.  We define filtration $\suite{\F_t}{t \geq 0}$ by $\F_t \eqdef \sigma(\eta_0,\zeta_0,...,\zeta_t)$. Verifying \ref{aF:01:decomp}, \ref{aF:01:x_*}, and \ref{aF:01:eps} is straightforward; only the bound \rf{phiP} of Assumption \ref{aF:01:phi} requires some effort.
%Notice that $f$ is a quadratic form: when denoting $b=\bE_{\eta\sim \pi}[r(\eta)\Phi(\eta)^\top]$ and $c=\half \bE_{\eta\sim \pi}[\|r(\eta)\|_2^2]$,
%\[
%f(x)=\half \la x\Sigma,x\ra - \la b,x\ra+c=\half \la (x-x_*)\Sigma, x-x_*\ra
%\]
%with the gradient
%\[g(x)=(x-x_*)\Sigma=x\bE_{\eta \sim \pi}\left[ \Phi(\eta) \Phi(\eta)^\top \right]-\bE_{\eta\sim \pi}[r(\eta)\Phi(\eta)^\top].
%\]
%Given observations $\eta_0,...,\eta_T$ we can build stochastic estimates (stochastic oracle) of $g(x)$,
%\[
%G(x,\underbrace{[\eta_t,\eta_{t+1}]}_{=:\omega_{t+1}})=(x\phi_t-\eta_{t+1})\phi_t^\top
%\]
%where $\phi_t=\Phi(\eta_t)$. Let us check that assumptions \ref{aF:01:decomp}--\ref{aF:01:eps} in the preceding section hold for the stochastic gradient $G(x,\omega_{t+1})$.
%
It is an immediate consequence of the following
\begin{lemma} \label{lem:vfar0}
Let \ref{vfar:zeta} and either \ref{vfar:r:item2} and \ref{vfar:phi:item1}, or \ref{vfar:r:item1} and \ref{vfar:phi}   hold true.
Then the Markov chain $\suite{\eta_t}{t\geq 0}$ has a unique invariant distribution which we denote $\pi$, and for all $i,j\leq n$,
\be \left|\bE_t\left[\Phi_i(\eta_{t+s}) \Phi_j(\eta_{t+s})   \right] - \bE_{\eta \sim \pi}\left[\Phi_i(\eta) \Phi_j(\eta)    \right]\right|\leq \rho_s \sN_t
\ee{new7}
with $\sup_t \bE[\sN_t^2]\leq\ov N^2<\infty$ and a geometrically decreasing positive rate function $\rho_s$.
\end{lemma}
The proof of the lemma is presented in Section \ref{sec:proof:vfar0}. Assume that $\rho_s$ as appearing in Lemma \ref{lem:vfar0} writes $\rho_s=\exp(-s/s_*)$ for all $s> 0$,
where $s_*>0$ is the corresponding ``half-life."
 In this case, for all
\[s\geq s_*\nu(T),\quad \nu(T):=\ln\left(\frac{D\ov NT}{ \sqrt{\ln(n+1)} r_0 \ov\phi^2+\sigma_*\ov\phi\sqrt{T}}\right),
\]
we have
\[
\rho_s\lesssim(D\ov N)^{-1}\left(\frac{\sqrt{\ln(n+1)} r_0 \ov\phi^2}{ T}+\frac{\sigma_* \ov\phi}{\sqrt{T}}\right),
\]
and we arrive at the following corollary of Proposition \ref{pro:bgamm1lr} as applied in the situation of this section.
\begin{proposition}\label{pro:far}
Suppose that Assumptions \ref{vfar:zeta}, \ref{vfar:X}, \ref{vfar:delta} and
either \ref{vfar:r:item2} and \ref{vfar:phi:item1}, or \ref{vfar:r:item1} and \ref{vfar:phi} hold true.  Under the $\ell_2/\ell_1$-proximal setup (see section~\ref{sec:appl:linearreg}), consider the SMD recursion \rf{prox1} with constant stepsizes $\gamma_t \eqdef\gamma>0$ for all $t \geq 1$, and assume that the initial condition $x_0 \in\Xset$ is such that $\bE[\norm{x_0-x_*}^2]\leq r_0^2$ and
 $s_*\nu(T)\in [2, T/2]$.
Let
$\gamma \eqdef \min\{\gamma_1,\gamma_2,\gamma_3\}$ with
\[\gamma_1 \eqdef \frac1{s_*\nu(T)\ov\phi^2},\quad
\gamma_2 \eqdef \sqrt{\frac{\ln(n+1)}T}  \frac{r_0}{\sigma_* \ov\phi},\quad
\gamma_3 \eqdef \left(\frac{\ln(n+1) r_0^2}{Ts_*^2\nu(T)^2 \sigma_*^2\ov\phi^4}\right)^{1/3}.
\]
Then
\[
\bE\left[f(\ov{x}_{T-\ovs}) \right] - f_* \lesssim  \frac{\ln(n+1) r_0^2 }{T} s_*\nu(T)\ov\phi^2+\sqrt{\frac{\ln(n+1)}T}r_0 \sigma_* \ov\phi+\left(\frac{\ln(n+1) r_0^2 s_*\nu(T)\sigma_* \ov\phi^{2}}T\right)^{2/3}.
\]
\end{proposition}

\paragraph{Block-sparse functional autoregression.}
Let us assume  \ref{hyp:sparse211}, i.e., that the parameter $x_*\in \bR^{p\times n}$ of the regression model \rf{etatr} possesses a column sparsity structure. Let us also assume  that $r(\cdot) = x_* \Phi(\cdot)$ (there is no bias) and that for some $\ul \varkappa>0$ $\bE_{\eta \sim \pi}\left[ \Phi(\eta) \Phi(\eta)^\top \right] \geq \ul \varkappa I$, impying
\ref{hyp:sparse211:2} and \ref{hyp:sparse211:3} hold true.
To recover the model parameter $x_*$, we use SMD-SR algorithm under $\ell_2/\ell_1$-proximal setup.
We have the following analogue of Corollary \ref{cor:sparser} in the present situation.
\begin{corollary}\label{cor:sparserar}
In the situation of this section, consider the SMD-SR algorithm with the initial condition $\ov y_{-1} \in\Xset$ satisfying $\norm{\ov y_{-1} -x_\star}\leq R$.
Denote
\[
T^*(\upsilon) \eqdef  \max\left\{\frac{\ln(n+1)\ov\phi^2 s_*\nu^2(\upsilon)\bar\chi}{\ul{\varkappa}},\;
\frac{\ln(n+1)\ov\phi^2\sigma_*s_*\nu(\upsilon)\bar\chi^{3/2}}{\ul\varkappa^{3/2} \upsilon},\;\frac{\ln(n+1)\varsigma_*^2\bar\chi^2}{\ul\varkappa^2 \upsilon^2}\right\}
\]
where $\nu(\upsilon) \eqdef \ln\left(2^{-k/2} \left[\sqrt{\ln(n+1)} D\ov N\ov\chi\right]^{-1}\ul\varkappa R\right)$.\\
The $\upsilon$-observation complexity of the SMD-SR algorithm is bounded by $T^*(\upsilon)$, up to an absolute constant. In other words,
the approximate solution $\wh x$ by SMD-SR algorithm utilizing $\cT$ observations \rf{etatr} satisfies
\[\bE\left[\|\wh x-x_*\|^2\right]\leq  \upsilon^2\]
for $\cT\gtrsim T^*(\upsilon)$.

\end{corollary}

%%%%%%%%%%
\appendix
%%%%%%%%%
\section*{Appendix}
\addcontentsline{toc}{section}{Annexes}
\section{Proofs for \Cref{sec:SMDproof}}
\label{app:proof:SMDproof}
We present here a general result on convergence of the SMD method under the assumptions in Section \ref{sec:probs}.
\begin{theorem}
\label{theo:a-sum1}
 Assume that assumptions \ref{hyp:problem}--\ref{hyp:oracle:Gt} hold.  Consider the recursion \rf{prox1} with constant stepsizes $\gamma_t = \gamma$, for all $t \geq 1$, with $\gamma \in(0, \left[\cL(1+2\kappa)\right]^{-1})$, and let for $t=1,2,...$,
  $\ov x_t \eqdef t^{-1} \sum_{i=0}^{t-1} x_i$.
  Then for all $T\geq 2$ and $s\in(0, T-1)$ and all $\F_0$-measurable $x_0 \in\Xset$ such that $\bE[\norm{x_0-x_*}^2]\leq r_0^2$,
\begin{align*}
    \bE\left[f(\ov{x}_{T-s})\right] -f_* &\leq \frac{1}{T-s} \sum_{t=0}^{T-s-1} \bE\left[f(x_t)\right]-f(x_*) \leq
    \frac{1}{T-s} \sum_{t=0}^{T-s-1} \bE\left[\la g(x_t),x_t-x_*\ra\right]
    \leq \mr_T(\gamma,s)
\end{align*}
where
\begin{multline} \label{mth1}
 \mr_T(\gamma,s)  \eqdef \varkappa_T\left(\frac{\Theta r_0^2}{\gamma T} + 2 (\varsigma_*^2 + \delta^2) \gamma +\lambda_T\delta\right)\left(3+6\gamma^2s^2\ell^2\cL^2(1+\kappa)\right)+
6\gamma^2s^2\ell^2\cL(\varsigma_*^2+\delta^2) \\
 +3\rho_s\lambda_T\left( \ov M + D \ov N \right)
\end{multline}
and
\begin{subequations}\label{lamT}
 \begin{align}
  \varkappa_T  & \eqdef [1-\gamma\cL(1+2\kappa)]^{-1}  \frac{T}{T-s},\label{lamT.a}\\
  \lambda_T&\left[=\lambda_T(r_0,\varsigma_*,\delta,\gamma)\right] \eqdef \sqrt{2 \Theta} r_0   + \tfrac{4}{3} \gamma (\varsigma_* +\delta)  \sqrt{T}  + \tfrac{1}{2} \delta \gamma T.\label{lamT.b}
  \end{align}
  \end{subequations}
\end{theorem}
\subsection{Proof of Theorem \ref{theo:a-sum1}}
We split the proof of the theorem into a sequence of technical lemmas. Recall that $G(x,\omega_{t+1})$ admits the following decomposition \aic{for any $x\in \Xset^o$}{}
\be
G(x_t,\omega_{t+1})=G_{t}(x_t)+\xi_{t+1}+\delta_{t}
\ee{xi*}
 and
 \begin{align*}
    \xi_{t+1}  \eqdef G(x_{t},\omega_{t+1})  - \bE_t\left[ G(x_{t},\omega_{t+1}) \right],\qquad
    \delta_{t} \eqdef
   \bE_t\left[ G(x_t,\omega_{t+1}) \right]-G_t(x_t).
\end{align*}
Denote
\[
\Delta_t \eqdef x_t -x_*.
\]
Throughout the proof,  $\vartheta_{x_0}$ and $V_{x_0}$ are denoted $\vartheta$ and $V$ respectively.
%By convexity of $f$
%\[
%f(\ov{x}_T)-f_* \leq \frac{1}{T} \sum_{t=0}^{T-1} f(x_t) - f_* \leq \frac{1}{T} \sum_{t=0}^{T-1}  \inner{g(x_t)}{\Delta_t}.
%\]
%First, we show in \Cref{prop:control:g} that the right-hand side above is upper bounded, in expectation, by a function of $\bE\left[ \norm{\Delta_t}\right]$, $\bE\left[ \norm{\Delta_t}^2 \right]$  and  $\bE\left[\inner{ G_{t}(x_{t})}{\Delta_{t}}\right]$ for $t=0, \cdots, T-1$.  Then  a control of these quantities is provided (see  \Cref{prop:control:deltaV} and \Cref{prop:control:Deltat}), and we conclude the proof.
%We will often use the following remark:
%\begin{remark} \label{rem:cocoercivity}
%By \Cref{hyp:oracle:decomp}-\cref{hyp:oracle:decomp:cocoercivity} and \Cref{hyp:oracle:Gt}-\cref{hyp:oracle:Gt:lipschitz}, the function   $x \mapsto G_{t}(x)$ is cocoercive; and  upon noting that $G_{t}(x_*) =0$,  it holds for any $t \geq 0$
%\[
%\normd{G_{t}(x)}^2 \leq  L_G \inner{G_{t}(x)}{x-x_*}, \quad \forall x \in \F_{t}.
%\]
%This implies that $\inner{G_{t}(x)}{x-x_*} \geq 0$ for all $x \in \F_t$.
%\end{remark}

\begin{lemma}
\label{prop:control:deltaV}
Let \ref{hyp:oracle:decomp} and \ref{hyp:oracle:Gt:lipschitz} hold. Then for any $\gamma \in (0, \left[ \cL(1+2\kappa) \right]^{-1})$ and $T\geq 1$,
\begin{align}
  \sum_{t=0}^{T-1}\bE\left[\inner{ G_t(x_t)}{\Delta_t} \right]   & \leq [\gamma (1-\gamma\cL(1+2\kappa))]^{-1}\left(\bE\left[V(x_0,x_*)\right]-\bE\left[V(x_{T},x_*)\right]\right) \nn
&\quad +[1-\gamma\cL(1+2\kappa)]^{-1}\left[\delta\sum_{t=0}^{T-1}\left( \bE\left[\norm{\Delta_t}^2\right] \right)^{1/2}+2(\varsigma_*^2+\delta^2)\gamma T\right].\label{proposition32}
\end{align}
\end{lemma}

\paragraph{Proof.}
Let $t \geq 0$, and let $0< \gamma <\left[\cL(1+2\kappa) \right]^{-1}$. The following three-point identity (see, e.g., \cite{chen1993convergence}) is the direct consequence of the optimality conditions for \rf{prox}:
\be \forall x\in \Xset\qquad
\gamma \inner{ G(x_t,\omega_{t+1})}{x_{t+1}-x}  \leq V(x_t,x)-V(x_{t+1},x)-V(x_t,x_{t+1}).
\ee{3v1}
Using \rf{xi*} and setting $x=x_*$ we get
\begin{align}
\gamma \inner{ G_t(x_t)}{\Delta_t} &\leq V(x_t,x_*)-V(x_{t+1},x_*)-V(x_t,x_{t+1})\nn &\quad -\gamma\inner{ \xi_{t+1}+\delta_{t}}{\Delta_t}
+\gamma\inner{ G_t(x_t)+\xi_{t+1}+\delta_{t}}{x_t-x_{t+1}}.\label{eq:tool1}
\end{align}
On the other hand,
\begin{align*}
&\gamma\inner{ G_t(x_t)+\xi_{t+1}+\delta_{t}}{x_t-x_{t+1}} \leq \gamma\normd{G_t(x_t)+\xi_{t+1}+\delta_{t}} \norm{x_t-x_{t+1}}\\
&\leq \tfrac{\gamma^2}2\normd{G_t(x_t)+\xi_{t+1}+\delta_{t}}^2+\half \norm{x_t-x_{t+1}}^2\leq \gamma^2\left(\normd{G_t(x_t)}^2+2\normd{\xi_{t+1}}^2+2\normd{\delta_{t}}^2\right)+\half \norm{x_t-x_{t+1}}^2.
\end{align*}
By the strong convexity of $\vartheta(\cdot)$, $V(x_t,x_{t+1})\geq \half \|x_{t+1}-x_t\|^2$; hence,
\begin{align*}
\gamma \inner{ G_t(x_t)}{\Delta_t} &\leq V(x_t,x_*)-V(x_{t+1},x_*) \\&\quad -\gamma\inner{ \xi_{t+1}+\delta_{t}}{\Delta_t}+ \gamma^2 \normd{G_t(x_t)}^2 + 2\gamma^2 \|\xi_{t+1}\|_*^2+2\gamma^2\|\delta_t\|_*^2.
\end{align*}
When taking the expectation on both sides, we obtain due to   \ref{hyp:oracle:decomp:xi}, \ref{hyp:oracle:decomp:delta}, and the Cauchy-Schwarz inequality,
\begin{align}
\gamma  \bE\left[\inner{ G_t(x_t)}{\Delta_t} \right] &\leq \bE\left[V(x_t,x_*)\right]-\bE\left[V(x_{t+1},x_*) \right]+ \gamma \delta \left(  \bE\left[\norm{\Delta_t}^2  \right]\right)^{1/2} \nn
& + \gamma^2  (1+2\kappa)\bE\left[\normd{G_t(x_t)}^2 \right]+ 2 \gamma^2 \left( \varsigma_*^2 + \delta^2\right).
\end{align}
Note that,  by \ref{hyp:oracle:decomp:cocoercivity}  and \ref{hyp:oracle:Gt:lipschitz}, $G_t(\cdot)$ is a Lipschitz-continuous gradient field of a convex function on $\Xset^o$. Thus,
\[
\normd{G_t(x_t)}^2  \leq \cL \inner{ G_t(x_t)-G_t(x_*)}{\Delta_t}=\cL \inner{ G_t(x_t)}{\Delta_t}
\]
due to $G_t(x_*) =0$ (see \ref{hyp:oracle:decomp:cocoercivity}).
As a consequence, \begin{align}
  \gamma \left(1 - \gamma\cL(1+2\kappa) \right)\bE\left[\inner{ G_t(x_t)}{\Delta_t} \right]   & \leq \bE\left[V(x_t,x_*)\right]-\bE\left[V(x_{t+1},x_*)\right] \nn
& +\gamma \delta \left( \bE\left[\norm{\Delta_t}^2\right] \right)^{1/2}+2\gamma^2 (\ \varsigma_*^2+\delta^2),\label{sum10}
\end{align}
and to show \rf{proposition32} it suffices now to sum the inequalities in \rf{sum10} from $t=0$ to $t=T-1$.\qed

%%%%%%%%%%%%%%%%

\begin{lemma}\label{prop:control:g} Assume \ref{hyp:oracle:decomp}  and \ref{hyp:oracle:Gt}.
For any $T \geq 1$,  $s \geq 0$ and $\beta>0$,
\begin{align*}
\left(1- \beta \right) \sum_{t=0}^{T-1} \bE\left[ \inner{g(x_{t})}{\Delta_t} \right]   \leq
 \left(1+ \beta \right)   \rho_s  \sum_{t=0}^{T-1} \bE\left[\sM_t \norm{\Delta_t} +\sN_t\norm{\Delta_t}^2 \right]   +
 \frac{c_\delta}{2\beta } \ell^2 \cL \left( \ \varsigma_*^2+ \delta^2 \right) \gamma^2 s^{2}  T \\
 + \left( 1+ \beta\ell^{-2}
 +  \frac{c_\delta}{2 \beta}\ell^2 \cL^2 (1+\kappa)  \gamma^2 s^{2}\right)  \sum_{t=0}^{T+s-1}  \bE\left[ \inner{G_{t}(x_{t})}{\Delta_{t}}\right], \label{lem31ex}
 \end{align*}
 where $c_\delta = 2$ if $\delta=0$ and $c_\delta =3$ otherwise.
\end{lemma}

\paragraph{Proof.} {\bf 1$^o$.}
Let $t,s \geq 0$ be fixed.
Due to the strong convexity of the d.-g.f. $\vartheta$, the prox-mapping \rf{prox} is Lipschitz-continuous w.r.t. to $\|\cdot\|_*$ with unit Lipschitz constant (cf., e.g., \cite[Lemma 1]{nesterov2009primal}) so that
\[\norm{x_{t+1} - x_t} \leq \gamma \normd{G(x_t,\omega_{t+1})}.
\] Hence, we have by  \rf{xi*}\footnote{Here, by convention, $\sum_{j=i}^{i-1} a_j=0.$}
\begin{align*}\norm{x_t-x_{t+s}} ^2&\leq\left(\sum_{j=1}^s\norm{x_{t+j}-x_{t+j-1}} \right)^2\leq
  \left(\sum_{j=1}^s\gamma \normd{G(x_{t+j-1},\omega_{t+j})} \right)^2\\&\leq \gamma^2\left(\sum_{j=1}^s  \left\{ \normd{G_{t+j-1}(x_{t+j-1})}+ \normd{\xi_{t+j}} + \normd{\delta_{t+j-1}} \right\} \right)^2 \\
  & \leq c_\delta \gamma^2 s \left( \sum_{j=0}^{s-1}\normd{G_{t+j}(x_{t+j})}^2+\sum_{j=1}^{s}\normd{\xi_{t+j}}^2 +\sum_{j=0}^{s-1} \normd{\delta_{t+j}}^2\right);
\end{align*}
$c_\delta =3$ when $\delta >0$ and $c_\delta = 2$ otherwise.
Equivalently,
\[
\norm{x_t-x_{t+s}}  \leq \gamma \sqrt{c_\delta s}  \left( \sum_{j=0}^{s-1}\normd{G_{t+j}(x_{t+j})}^2+ \sum_{j=1}^{s}\normd{\xi_{t+j}}^2 +\sum_{j=0}^{s-1} \normd{\delta_{t+j}}^2\right)^{1/2}.
\]

 \paragraph{2$^o$.} %We now return to the proof of the lemma.
%%
%
%\end{proof}
%\begin{proof}{(of \Cref{prop:control:g})}
Let $t \geq 0$  and $s \geq 0$ be  fixed.  One has
\begin{align}
\inner{g(x_{t})}{\Delta_t}  &=
\inner{g(x_t)-\bE_t\left[G_{t+s}(x_t)\right]}{\Delta_t}+\bE_t\left[\inner{ G_{t+s}(x_{t+s})}{\Delta_{t+s}}\right] \nn
&\qquad + \bE_t\left[\inner{G_{t+s}(x_t)- G_{t+s}(x_{t+s})}{\Delta_t}+\inner{ G_{t+s}(x_{t+s})}{\Delta_t -\Delta_{t+s}} \right]  .
\label{gxt}
\end{align}
By \ref{hyp:oracle:Gt:ell}, for any $\alpha >0$, \footnote{Without ``symmetry condition'' \rf{ellsym1}, only using the Lipschitz continuity of $G_t(\cdot)$ one has
\[\la G_{t+s}(x_t)-G_{t+s}(x_{t+s}),\Delta_t\ra \leq  \cL\|x_{t+s}-x_t\|\,\|\Delta_t\|,
\] leading to a significantly less precise bound. }
\begin{align*}
&\inner{ G_{t+s}(x_t)-G_{t+s}(x_{t+s})}{\Delta_t} +\inner{ G_{t+s}(x_{t+s})}{\Delta_t-\Delta_{t+s}} \leq \norm{x_t-x_{t+s}}  \left(
\ell \normd{G_{t+s}(x_t)} +\normd{G_{t+s}(x_{t+s})}  \right) \\
&\qquad \leq \frac{\gamma}{2\alpha}\ \sqrt{c_\delta s}
\left( \sum_{j=0}^{s-1}\normd{G_{t+j}(x_{t+j})}^2+ \sum_{j=1}^{s}\normd{\xi_{t+j}}^2 +\sum_{j=0}^{s-1} \normd{\delta_{t+j}}^2\right) \\
& \qquad + \alpha\gamma \sqrt{c_\delta s} \left(\ell^2\normd{G_{t+s}(x_t)}^2+\normd{G_{t+s}(x_{t+s})}^2\right)% \label{eq:tool4}
\end{align*}
where  the last ``$\leq $'' is due to the straightforward inequality
\[a ( \ell b +c) \leq \tfrac{1}{2 \alpha} a^2 + \tfrac{\alpha}{2}  (\ell b +c)^2 \leq \tfrac{1}{2\alpha}  a^2 +  \alpha (\ell^2 b^2 +c^2).\]   Observe that, due to the cocoercivity of $G_t(\cdot)$ on $\Xset$ (consequence of \ref{hyp:oracle:decomp:cocoercivity}),
\begin{align}
\normd{G_{t+s}(x_t)}^2 & \leq  \cL \inner{G_{t+s}(x_t)}{\Delta_t} \leq \cL \inner{G_{t+s}(x_t) - g(x_t)}{\Delta_t}+ \cL \inner{g(x_t)}{\Delta_t}, \nn
\normd{G_{t+s}(x_{t+s})}^2 &\leq  \cL \inner{ G_{t+s}(x_{t+s})}{\Delta_{t+s}}, \label{eq:cocoerciv2}
\end{align}  %ee{eq:tool5}
while, by \ref{hyp:oracle:Gt:mixing},
\[%be
\left| \inner{ g(x_t)-\bE_t\left[G_{t+s}(x_t)\right]}{\Delta_t} \right| \leq\normd{g(x_t)-\bE_t\{G_{t+s}(x_t)\}} \norm{\Delta_t} \leq  \rho_s (\sM_t + \sN_t \norm{\Delta_t})\norm{\Delta_t}.
\]%ee{ex20}
Substituting the above bounds into \rf{gxt}  yields
\begin{align*}
\left( 1 - \alpha\gamma \sqrt{c_\delta s} \ell^2  \cL \right) \inner{g(x_{t})}{\Delta_t}
&\leq  \frac{\gamma}{2\alpha} \sqrt{{c_\delta}s}  \bE_t\left[ \sum_{j=0}^{s-1}\normd{G_{t+j}(x_{t+j})}^2+ \sum_{j=1}^{s}\normd{\xi_{t+j}}^2 +\sum_{j=0}^{s-1} \normd{\delta_{t+j}}^2\right]  \\
& +
 (1+ \alpha \gamma \sqrt{c_\delta s} \ell^2  \cL)   \rho_s (\sM_t \norm{\Delta_t} +\sN_t\norm{\Delta_t}^2)
  \\
 & + (1+\alpha  \gamma \sqrt{c_\delta s}\cL ) \bE_t\left[\inner{ G_{t+s}(x_{t+s})}{\Delta_{t+s}}\right],
\end{align*}
and when taking the expectation, using \ref{hyp:oracle:decomp:xi}, \ref{hyp:oracle:decomp:delta}, and \rf{eq:cocoerciv2}, we obtain
\begin{align*}
\left( 1 - \alpha \gamma \sqrt{{c_\delta}s} \ell^2  \cL \right) & \bE\left[ \inner{g(x_{t})}{\Delta_t} \right]   \leq
 (1+ \alpha \gamma \sqrt{{c_\delta}s} \ell^2  \cL)   \rho_s  \bE\left[\sM_t \norm{\Delta_t} +\sN_t\norm{\Delta_t}^2 \right]  \\
 & + (1+\alpha\gamma \sqrt{{c_\delta}s} \cL) \bE\left[\inner{ G_{t+s}(x_{t+s})}{\Delta_{t+s}}\right]\\ &
 + \tfrac{\gamma}{2\alpha} \sqrt{c_\delta} s^{3/2}  \left( \ \varsigma_*^2+ \delta^2 \right)
 +  \tfrac{\gamma}{2\alpha} \sqrt{c_\delta s}\,\cL(1+\kappa) \sum_{j=t}^{t+s-1}\bE\left[ \inner{G_{j}(x_{j})}{\Delta_{j}}\right].
\end{align*}
Finally, when summing up from $t=0$ to $T-1$, we get
\begin{multline*}
\left( 1 -  \alpha\gamma \sqrt{{c_\delta}s} \ell^2  \cL\right) \sum_{t=0}^{T-1} \bE\left[ \inner{g(x_{t})}{\Delta_t} \right]   \leq
 \left( 1 + \alpha \gamma \sqrt{{c_\delta}s} \ell^2  \cL \right)  \rho_s  \sum_{t=0}^{T-1} \bE\left[\sM_t \norm{\Delta_t} +\sN_t\norm{\Delta_t}^2 \right]  \\
 + (1+\alpha \gamma \sqrt{{c_\delta}s} \cL ) \sum_{t=s}^{T+s-1}\bE\left[\inner{ G_{t}(x_{t})}{\Delta_{t}}\right]\\
 + \tfrac{\gamma}{2\alpha} \sqrt{c_\delta} s^{3/2}  T \left( \ \varsigma_*^2+ \delta^2 \right)
 + \tfrac{\gamma}{2\alpha} \sqrt{c_\delta} s^{3/2} \cL(1+\kappa)  \sum_{t=0}^{T+s-2}  \bE\left[ \inner{G_{t}(x_{t})}{\Delta_{t}}\right].
\end{multline*}
We conclude by setting $\beta \eqdef \alpha \gamma \sqrt{{c_\delta}s} \ell^2  \cL$, or $\alpha \eqdef \left(\gamma \sqrt{{c_\delta}s} \ell^2  \cL\right)^{-1}\beta$ when $s>0$.  \qed
%\end{proof}

%%%%%%%%%%%%%%%%%%%%%%%%%%%%%%%%%%%
%%%%%%%%%%%%%%%%%%%%%%%%%%%%%%%%%%%

\medskip

\noindent
\begin{lemma}\label{prop:control:Deltat} Let \ref{hyp:oracle:decomp} and \ref{hyp:oracle:Gt:lipschitz} hold.
    Let also $\gamma \in (0, \left[  \cL(1+2\kappa)\right]^{-1})$ and $T \geq 1$. Then for any initial condition $x_0 \in\Xset$ such that $\bE\left[ \norm{x_0-x_*}^2\right] \leq r_0^2$
\begin{align}
& \sum_{t=0}^{T-1} \left(\bE\{\norm{\Delta_t}^2\} \right)^{1/2} \leq  \sqrt{2 \Theta}  r_0 T + \tfrac{4}{3} \gamma (\varsigma_*+\delta)  T^{3/2} + \tfrac{1}{2}\delta \gamma T^2, \label{lem3e1}\\
& \sum_{t=0}^{T-1}  \bE\{\norm{\Delta_t}^2\} \leq  \left\{
    \begin{tabular}{c c}
      $2 \Theta r_0^2 T + \tfrac{8}{3} \gamma \sqrt{2} \sqrt{\Theta} r_0 \varsigma_* T^{3/2} + 2 \gamma^2 \varsigma_*^2 T^2$   &  \text{if $\delta =0$}, \\
    $6 \Theta r_0^2 T + 6 \gamma^2 (\varsigma_*^2 +\delta^2) T^2 +  \delta^2 \gamma^2 T^3 $    &  \text{otherwise.} \\
    \end{tabular} \right.\nonumber
\end{align}
\end{lemma}
\paragraph{Proof.} Denote $v^2_{t} \eqdef  2\bE\left[ V(x_t,x_*)\right]$. Note that \eqref{qg1} yields
\[
v_0^2 = 2 \bE\left[V(x_0, x_*) \right] \leq 2\Theta \bE\left[ \norm{x_0 -x_*}^2\right].
\]Under the stated assumptions, results from the proof of \Cref{prop:control:deltaV} apply: \rf{sum10} holds. Under \ref{hyp:oracle:decomp:cocoercivity}, the left-hand side in  \rf{sum10} is nonnegative. Thus, for all
$t \geq 0$
\[ \bE\left[V(x_{t+1},x_*)\right] \leq \bE\left[V(x_t,x_*)\right]
+ \delta \gamma \left(\bE\left[\norm{\Delta_t}^2\right] \right)^{1/2}+2\gamma^2(\varsigma_*^2+\delta^2).
\]
Because $\bE\{\norm{\Delta_t}^2\}\leq v_t^2$,  this implies that
\[
v^2_{t+1}\leq v^2_t+2\delta\gamma \left( \bE\{\norm{\Delta_t}^2\} \right)^{1/2}+4\gamma^2(\varsigma_*^2+\delta^2)\leq v^2_t+2\delta\gamma v_t+ 4\gamma^2(\varsigma_*^2+\delta^2).
\]
Let us show that
\be
v_t \leq \bar v_t \eqdef  v_0 +  \delta \gamma t  + 2 \gamma \sqrt{\varsigma_*^2 + \nicefrac{3}{4}\, \delta^2} \sqrt{t}.
\ee{recursion}
Observe that \rf{recursion} clearly holds true for $t=0$. Assume that it holds for some $t\geq 0$; let us show that  $\bar v_{t+1}^2 \geq  \bar v_t^2 +  2\delta\gamma \bar v_t + 4\gamma^2(\varsigma_*^2+\delta^2)$ what would imply that \rf{recursion}  also holds for $t+1$. We write
\begin{align*}
\bar v_{t+1}^2 - \bar v_t^2 &= 4 \gamma^2 (\varsigma_*^2 + \nicefrac{3}{4}\, \delta^2) + \delta \gamma \left(2(v_0 + \delta \gamma t) + \delta \gamma \right) + 4 \delta \gamma^2 \sqrt{\varsigma_*^2 + \nicefrac{3}{4}\,\delta^2} \sqrt{t+1} \\
& + 4 \gamma \sqrt{\varsigma_*^2 + \nicefrac{3}{4}\,\delta^2} (v_0 + \delta \gamma t) \left(\sqrt{t+1} - \sqrt{t} \right) \\
&\geq 4 \gamma^2 (\varsigma_*^2 +  \delta^2) + 2 \delta \gamma \left(v_0 + \delta \gamma t + 2 \gamma  \sqrt{\varsigma_*^2 + \nicefrac{3}{4}\,\delta^2} \sqrt{t}  \right) \end{align*}
which concludes the proof.
 Therefore,
\begin{align}
\bE[\norm{\Delta_t}^2]\leq v_t^2 &\leq \left(v_0 +  \delta \gamma t  + 2 \gamma \sqrt{\varsigma_*^2 + \nicefrac{3}{4}\, \delta^2} \sqrt{t} \right)^2  \leq \left(v_0 +  \delta \gamma t  + 2 \gamma \sqrt{\varsigma_*^2 +  \delta^2} \sqrt{t} \right)^2  \label{eq:majovt2} \\
&= v_0^2 + 4 \gamma v_0 \sqrt{\varsigma_*^2 + \delta^2} \sqrt{t}+  [2 \gamma v_0 \delta + 4 \gamma^2 (\varsigma_*^2 + \delta^2)] t + 4 \gamma^2 \sqrt{\varsigma_*^2 + \delta^2}\delta t^{3/2}  + \delta^2 \gamma^2 t^2. \nonumber
\end{align}
When $\delta=0$, this yields
\[
\bE[\norm{\Delta_t}^2] \leq v_0^2 + 4 \gamma v_0 \varsigma_* \sqrt{t}+ 4 \gamma^2 \varsigma_*^2  t;
\]
and summing from $t=0$ to $t =T-1$,
\[
\sum_{t=0}^{T-1}\bE[\norm{\Delta_t}^2]\leq v_0^2 T + \tfrac{8}{3} \gamma v_0 \varsigma_* T^{3/2} + 2 \gamma^2 \varsigma_*^2 T^2.
\]
When $\delta \neq 0$,  by using $2 ab \leq a^2 + b^2$, we obtain
\[
\bE[\norm{\Delta_t}^2]  \leq 3 v_0^2  + 12 \gamma^2 (\varsigma_*^2 + \delta^2) t+ 3 \delta^2 \gamma^2 t^2.
\]
and summing from $t=0$ to $t =T-1$,
\[\sum_{t=0}^{T-1}\bE[\norm{\Delta_t}^2]\leq  3 v_0^2 T + 6 \gamma^2 (\varsigma_*^2 +\delta^2) T^2 +  \delta^2 \gamma^2 T^3;
\]
From \eqref{eq:majovt2}, it holds
\[
\bE[\norm{\Delta_t}^2]^{1/2} \leq v_0 + \delta\gamma t+2 \gamma \sqrt{\varsigma_*^2+\delta^2} \sqrt{t}\leq v_0 + \delta\gamma t+2 \gamma (\varsigma_*+\delta) \sqrt{t}.
\]
Summing from $t=0$ to $t =T-1$ we have
\[
\sum_{t=0}^{T-1} \bE[\norm{\Delta_t}^2]^{1/2} \leq  v_0 T + \tfrac{4}{3} \gamma (\varsigma_*+\delta)  T^{3/2} + \tfrac{1}{2}\delta \gamma T^2.\eqed
\]

\paragraph{Proof of Theorem \ref{theo:a-sum1}}
\label{sec:proof:theo:sum1}
%Denote $\Delta_t= x_t -x_*$; note that $\norm{\Delta_t} \leq D$.
By convexity of $f$, we have
\[
f(\ov{x}_{T-s}) - f_*  \leq  \frac{1}{T-s}\sum_{t=0}^{T-s-1}f(x_t) - f_* \leq \frac{1}{T-s} \sum_{t=0}^{T-s-1} \inner{g(x_t)}{\Delta_t}.
\]
Applying \Cref{prop:control:g} with $\beta=1/2$  (note that $T-s \geq 1$ and  $s \geq 1$),  we get
\begin{align*}
 \sum_{t=0}^{T-s-1} \bE\left[ \inner{g(x_{t})}{\Delta_t} \right]   \leq 3  \rho_s  \sum_{t=0}^{T-s-1} \bE\left[\sM_t \norm{\Delta_t} +\sN_t\norm{\Delta_t}^2 \right]   +  6\gamma^2 s^{2}  \ell^2 \cL  (\varsigma_*^2 + \delta^2)  (T-s) \\
 + \left( 3+6 \gamma^2 s^2 \ell^2 \cL^2(1+\kappa)    \right) \sum_{t=0}^{T-1}  \bE\left[ \inner{G_{t}(x_{t})}{\Delta_{t}}\right];
 \end{align*}
 (we took into account that $\ell\geq 1$ by \ref{hyp:oracle:Gt:ell}).
On the other hand, using
the bounds $V(x_T, x_*) \geq 0$ and $0\leq\bE\left[V(x_0, x_*) \right] \leq \Theta \bE\left[ \norm{x_0 - x_*}^2 \right] \leq \Theta r_0^2$,
we conclude by Lemma \ref{prop:control:deltaV} that
\[
  \sum_{t=0}^{T-1}\bE\left[\inner{ G_t(x_t)}{\Delta_t} \right]    \leq [1-\gamma\cL(1+2\kappa)]^{-1}\left({\gamma^{-1}\Theta r_0^2}+ 2 (\varsigma_*^2 + \delta^2) \gamma T + \delta \sum_{t=0}^{T-1} \left( \bE\left[ \norm{\Delta_t}^2\right] \right)^{1/2}\right).
  \]
Furthermore,
\begin{align*}
 \sum_{t=0}^{T-s-1} \bE\left[\sM_t \norm{\Delta_t} +\sN_t\norm{\Delta_t}^2 \right] & \leq \ov M  \sum_{t=0}^{T-s-1} \left(\bE\left[\norm{\Delta_t}^2 \right] \right)^{1/2} +  \ov N \sum_{t=0}^{T-s-1} \left(\bE\left[\norm{\Delta_t}^4 \right] \right)^{1/2} \\
 & \leq \left( \ov M + D \ov N \right) \sum_{t=0}^{T-s-1} \left(\bE\left[\norm{\Delta_t}^2 \right] \right)^{1/2}
\end{align*}
due to the upper bound $\norm{\Delta_t} \leq D$. By \Cref{prop:control:Deltat},
\begin{align*}
 \sum_{t=0}^{T-s-1} \bE\left[\sM_t \norm{\Delta_t} +\sN_t\norm{\Delta_t}^2 \right]
 \leq \left( \ov M + D \ov N \right) \left( \sqrt{2 \Theta} r_0  (T-s)+ \tfrac{4}{3} \gamma (\varsigma_* +\delta)  (T-s)^{{3}/{2}}  + \tfrac{1}{2} \delta \gamma (T-s)^2\right).
\end{align*}
This yields
\begin{align*}
 & \frac{1}{T-s} \sum_{t=0}^{T-s-1}  \bE\left[ \inner{g(x_{t})}{\Delta_t} \right]  \leq
 3 [1-\gamma \cL(1+2\kappa)]^{-1}\frac{T}{T-s}   \ \left( \frac{\Theta r_0^2}{\gamma T} + 2 (\varsigma_*^2 + \delta^2) \gamma \right) \\&\qquad
  + 6\gamma^2 s^2 \ell^2 \cL (\varsigma_*^2 + \delta^2) \\
 &\qquad + 6\gamma^2 s^2 \ell^2 \cL^2 (1+\kappa) [1-\gamma\cL(1+2\kappa)]^{-1}\frac{T}{T-s}   \left( \frac{\Theta r_0^2}{\gamma T} + 2 (\varsigma_*^2 + \delta^2) \gamma \right) \\
 &\qquad + 3 \rho_s \left( \ov M + D \ov N \right) \left( \sqrt{2 \Theta} r_0 + \tfrac{4}{3} \gamma (\varsigma_* +\delta)  (T-s)^{{1}/{2}}  + \tfrac{1}{2} \delta \gamma (T-s)\right) \\
 &\qquad +   3
  [1-\gamma\cL(1+2\kappa)]^{-1} \frac{T}{T-s} \delta \left(\sqrt{2 \Theta} r_0   + \tfrac{4}{3} \gamma (\varsigma_* +\delta)  T^{{1}/{2}}  + \tfrac{1}{2} \delta \gamma T \right) \\
 &\qquad +   6\gamma^2 s^{2}  \ell^2 \cL^2 \left(1+\kappa \right)    [1-\gamma\cL(1+2\kappa)]^{-1} \frac{T}{T-s} \delta \left(\sqrt{2 \Theta} r_0   + \tfrac{4}{3} \gamma (\varsigma_* +\delta)  T^{{1}/{2}}  + \tfrac{1}{2} \delta \gamma T \right)\\
 &\leq \varkappa_T\left(\frac{\Theta r_0^2}{\gamma T} + 2 (\varsigma_*^2 + \delta^2) \gamma +\lambda_T\delta\right)\left(3+6\gamma^2s^2\ell^2\cL^2(1+\kappa)\right)\\&\qquad+
 6\gamma^2s^2\ell^2\cL(\varsigma_*^2+\delta^2)+3\rho_s\lambda_T\left( \ov M + D \ov N \right),
 \end{align*}
implying the statement of the theorem. \qed
\subsection{Proof of Proposition \ref{prop:iid:case}} \label{sec:proof:prop:iid:case}
Bound \rf{indc1} is an immediate consequence of Lemmas \ref{prop:control:deltaV} and \ref{prop:control:Deltat}. Indeed, when denoting $\Delta_t = x_t -x_*$ and recalling that $\bE[V(x_0,x)]\leq \Theta r_0^2$, and that, in the present situation, $G_t(x_t)=g(x_t)$,  we get from \rf{proposition32} and \rf{lem3e1}
\begin{align*}
  \sum_{t=0}^{T-1}\bE\left[\inner{ g(x_t)}{\Delta_t} \right]   & \lesssim \gamma^{-1} \bE[V(x_0,x)]
  +\delta\sum_{t=0}^{T-1}\left( \bE\left[\norm{\Delta_t}^2\right] \right)^{1/2}+(\varsigma_*^2+\delta^2)\gamma T
  \\&\lesssim \gamma^{-1}\Theta r_0^2+\delta T \left[ \sqrt{\Theta} r_0+ \gamma (\varsigma_*+\delta) \sqrt{T} + \delta \gamma T\right]+
  (\varsigma_*^2+\delta^2)\gamma T.
%  \label{proposition312}
\end{align*}
Now, by convexity of $f$,
\begin{align*}
\bE[f(\ov x_T)]-f_*&\leq \frac 1T \sum_{t=0}^{T-1}(\bE[f(x_t)]-f_*)\leq \frac 1T \sum_{t=0}^{T-1}\bE\left[\inner{ g(x_t)}{\Delta_t} \right]\\&
\lesssim \frac{\Theta r_0^2}{\gamma T}+\gamma \varsigma_*^2+
\delta \sqrt{\Theta} r_0 \delta\gamma\varsigma_* \sqrt{T}+\delta^2 \gamma T.
  \end{align*}
The bound \rf{ssb0} is obtained by a straightforward substitution of \rf{deltaindep}.
To establish the ``in particular" part of the statement of the proposition, it suffices to notice that
the choice of $\gamma \eqdef  \left( \frac{r_0}{\varsigma_*}\sqrt{\frac{\Theta}T} \right) \wedge \frac{1}{2 \cL (1+2\kappa)}$ is balancing the terms on the right-hand side of \rf{ssb0} under the constraint that $\gamma \leq (2 \cL (1+2 \kappa))^{-1}$. When $\gamma=\frac{r_0}{\varsigma_*}\sqrt{\frac{\Theta}T} \leq [2\cL(1+2\kappa)]^{-1}$ (which implies that $\varsigma_*>0$), this yields the bound
\[\bE\left[f(\ov{x}_T) \right] - f_* \lesssim \sqrt{\frac{\Theta}T} r_0 \varsigma_*.
\] When $\gamma=[2\cL(1+2\kappa)]^{-1}\leq \sqrt{\frac{\Theta}T}\frac{r_0}{\varsigma_*}$, one clearly has $\varsigma_*^2\gamma\leq \frac{\Theta r_0^2}{\gamma T}$, implying that
\[ \bE\left[f(\ov{x}_T) \right] - f_*\lesssim  \frac{\Theta r_0^2 \cL(1+\kappa)}{T}.
\eqno{\mbox{\qed}}
\]
\subsection{Proof of Proposition \ref{pro:sum1}}\label{proof:sum1}
Note that \rf{biasv1}  and the condition $\gamma \leq [2 \cL (1+2\kappa)]^{-1}$ imply that $\delta \leq (\varsigma_* / \sqrt{T}) \vee ((\sqrt{\Theta} r_0)/(\gamma T))$.
Relationship \rf{bp1sbias} is an immediate consequence of the general bound of Theorem \ref{theo:a-sum1} under the small bias condition. Indeed, % suppose that
%\[\delta\lesssim \max\left[\frac{\Theta^{1/2}r_0}{\gamma T},\,\frac{\varsigma_*}{\sqrt{T}}\right].
%\]
let $\lambda_T$ be defined as in \rf{lamT.b}. One easily verifies that, when \rf{biasv1} holds,
\[
\lambda_T\lesssim\left(\sqrt{\Theta} r_0   + \gamma (\varsigma_* +\delta)  \sqrt{T}  + \delta \gamma T \right)\lesssim\left(\sqrt{\Theta} r_0   + \gamma \varsigma_*\sqrt{T}\right),
\]
and
 \begin{align*}
 \delta^2 \gamma & \lesssim \frac{\varsigma_*^2 \gamma}{T} + \frac{\Theta r_0^2}{\gamma T^2}, \\
 \lambda_T\delta&\lesssim\left(\sqrt{\Theta} r_0   + \gamma \varsigma_*\sqrt{T}\right)\left[\frac{\sqrt{\Theta}r_0}{\gamma T}+\frac{\varsigma_*}{\sqrt{T}}\right]
 \lesssim\frac{\Theta r^2_0}{\gamma T}+\gamma\varsigma_*^2.
 \end{align*}
Furthermore, when $T\geq 2s$ and $\gamma\leq [2\cL(1+2\kappa)]^{-1}$, $\varkappa_T$, as defined in \rf{lamT.a}, satisfies
$\varkappa_T\leq 4$, and the
bound \rf{mth1} of Theorem \ref{theo:a-sum1} becomes
\begin{align*}\bE\left[f(\ov{x}_{T-s})\right] -f_*
&\lesssim \left(\frac{\Theta r_0^2}{\gamma T} +\varsigma_*^2\gamma\right)
\big(1+\gamma^2s^2\ell^2\cL^2(1+\kappa)\big)+
 \gamma^2s^2\ell^2\cL\varsigma_*^2\\&\quad +\rho_s\left(\sqrt{\Theta} r_0   +\gamma \varsigma_*\sqrt{T}\right)\left( \ov M + D \ov N \right). \tag*{\qed}
\end{align*}
\subsection{Proof of Proposition \ref{pro:bgamm1}} \label{sec:proof:pro:bgamm1}
{\bf 1$^o$.} Let $\ovs \ell \leq \frac{2(1+2\kappa)}{\sqrt{1+\kappa}}$. Note that in this case, for $\gamma\leq [2\cL(1+2\kappa)]^{-1}$, one has
\[
\gamma^2\ovs^2\ell^2\cL^2(1+\kappa)\leq \frac{\ovs^2\ell^2(1+\kappa)}{4(1+2\kappa)^2}\leq 1,
\]
and
\[
\gamma\ovs^2\ell^2\cL\leq  \frac{\ovs^2\ell^2}{2(1+2\kappa)}\leq 4,
\]
implying that
\[
\bE\left\{f(\ov{x}_{T-\ovs})\right\} -f_* \lesssim
    \frac{\Theta r_0^2}{\gamma T} +\varsigma_*^2\gamma
\]
 for all $\gamma\leq [2\cL(1+2\kappa)]^{-1}$.
\paragraph{2$^o$.} We now assume that $\ovs \ell \geq \frac{2(1+2\kappa)}{\sqrt{1+\kappa}}$. Set $\gamma_{\max} \eqdef [2\cL(1+2\kappa)]^{-1}$. Consider first the case $\varsigma_*=0$. Then \eqref{norhos} gets into
\[
 \bE\left[f(\ov{x}_{T-\ovs})\right] -f_* \lesssim \frac{\Theta r_0^2}{\gamma T}
\big(1+\gamma^2\ovs^2\ell^2\cL^2(1+\kappa)\big).
\]
On the constraint set $\gamma \leq \gamma_{\max}$, the bound is minimized by setting $\gamma = \gamma_1$ where
\[
\gamma_1 \eqdef \frac{1}{\ov s \ell \cL \sqrt{1+\kappa}};
\] observe that $\gamma_1 \leq \gamma_{\max}$ since $\ov s \ell \geq  2(1+2\kappa)/\sqrt{1+\kappa}$. When $\gamma = \gamma_1$, the bound is  $\Theta r_0^2 \ov s \ell \cL \sqrt{1+\kappa}/T$. \\
Consider now the case $\varsigma_* >0$.
Select
$\gamma_1$, $\gamma_2$, and $\gamma_3$ in order to satisfy the following ``balance equalities":
\begin{align*}
\gamma_1&:\quad \frac{\Theta r_0^2}{\gamma_1 T}=\frac{\Theta r_0^2}{ T}\gamma_1\ovs^2\ell^2 \cL^2(1+\kappa)\left[=\frac{\Theta r_0^2}T \ovs\ell\cL\sqrt{1+\kappa}\right],\\
\gamma_2&:\quad \frac{\Theta r_0^2}{\gamma_2 T}=\gamma_2\varsigma_*^2\left[=\sqrt{\frac{\Theta}T}r_0 \varsigma_*\right],\\
\gamma_3&:\quad \frac{\Theta r_0^2}{\gamma_3 T}=\gamma^2_3\ovs^2\ell^2\cL\varsigma_*^2\left[=\cL^{1/3}\left(\frac{\Theta r_0^2 \ovs\ell\varsigma_*}T\right)^{2/3} \right].
\end{align*}
Set $\gamma_* \eqdef \min(\gamma_1,\gamma_2,\gamma_3)$. Note that $\gamma_1\leq \gamma_{\max}$, thus implying that $\gamma_* \leq \gamma_{\max}$.  By definition of $\gamma_1$, we have
\[
\gamma \leq \gamma_1 \Longrightarrow \frac{\Theta r_0^2}{\gamma T} \gamma^2 \ovs^2 \ell^2 \cL^2 (1+\kappa) \leq \frac{\Theta r_0^2}{\gamma T}.
\]
Similarly, it holds
\[
\gamma \leq \gamma_2 \wedge \gamma_{\max} \Longrightarrow \varsigma_*^2 \gamma \leq \frac{\Theta r_0^2}{\gamma T},  \qquad
\gamma \leq \gamma_3 \wedge \gamma_{\max} \Longrightarrow \gamma^2 \ovs^2 \ell^2 \cL \varsigma_*^2 \leq \frac{\Theta r_0^2}{\gamma T}.
\]
We conclude that when $\gamma_* = \gamma_1$,
\[
\bE\left[f(\ov{x}_{T-\ovs})\right] -f_* \lesssim \frac{\Theta r_0^2}{\gamma_1 T} = \frac{\Theta r_0^2}{T} \ovs \ell \cL \sqrt{1+\kappa};
\]
when $\gamma_* = \gamma_2$,
\[
\bE\left[f(\ov{x}_{T-\ovs})\right] -f_* \lesssim \varsigma_*^2 \gamma_2 = \frac{\sqrt{\Theta}}{\sqrt{T}} \varsigma_* r_0;
\]
when $\gamma_* = \gamma_3$,
\[
\bE\left[f(\ov{x}_{T-\ovs})\right] -f_* \lesssim \frac{\Theta r_0^2}{\gamma_3 T} = \cL^{1/3}  \left(\frac{\Theta r_0^2 \ovs \ell \varsigma_*}{T} \right)^{2/3}.
\]
This concludes the proof.

\subsection{Proof of Proposition~\ref{cor:complex1}}
\label{sec:proof:cor:complex1}
It suffices to check that the terms on the right-hand side of \rf{mth1} are bounded with $C \upsilon$ for some small absolute $C$.  Throughout the proof, we use $\ov s$ instead of $\ov s(\upsilon)$. Choose $T$ such that
\[
 T \eqdef \frac{\Theta r_0^2}{\upsilon \gamma}.
\]
The conditions $\ov s \ell > 2(1+2\kappa)/\sqrt{1+\kappa}$ and $\gamma \leq (\ov s \ell \cL \sqrt{1+\kappa})^{-1}$ imply that $\gamma <1/(2 \cL (1+2\kappa))$.  \\
In addition,   since $\upsilon \leq \Theta r_0^2 \cL/2$, then $\upsilon \leq \Theta r_0^2 \ell \cL \sqrt{1+\kappa}/2$; together with $\gamma \leq (\ov s \ell \cL \sqrt{1+\kappa})^{-1}$, this yields $\ov s \leq T/2$.  Since $\ov s \leq T/2$, we have $\varkappa_T \leq 4$.  \\
The conditions $\gamma \leq 1/(\ov s \ell \cL \sqrt{1+\kappa})$ and  $\upsilon \leq \Theta r_0^2 \cL/2$ imply
\[
3 + 6 \gamma^2 \ov s^2 \ell^2 \cL^2 (1+\kappa)  \lesssim 1, \qquad 6 \gamma^2 \ov s^2 \ell^2 \cL \delta^2 \leq 6 \gamma^2 \ov s^2 \ell^2 \cL^2 (c')^2 \upsilon \lesssim  \upsilon.
\]
We prove that $\gamma \delta^2 \lesssim \upsilon$. Since $\upsilon \leq \Theta r_0^2 \cL/2$, $\ov s \ell \sqrt{1+\kappa} > 2 (1+2\kappa)$ and
$\gamma \leq (\ov s \ell \cL \sqrt{1+\kappa})^{-1}$,
then the proof is concluded upon noting that
\[
\upsilon \leq \Theta r_0^2 \cL (1+2\kappa) \leq \frac{1}{2}\Theta r_0^2 \ov s \ell \cL \sqrt{1+\kappa} \leq \frac{\Theta r_0^2}{2 \gamma}.
\]
The condition $\gamma \ov s \ell \sqrt{\cL} \varsigma_* \leq \sqrt{\upsilon}$ implies that $\gamma^2 \ov s^2 \ell^2 \cL \varsigma_*^2 \leq \upsilon.$ \\
We now prove $\lambda_T \lesssim \sqrt{\Theta r_0}$.  The condition $\gamma \leq \upsilon / \varsigma_*^2$ implies that $\gamma \sqrt{T} \leq \sqrt{\Theta} r_0 / \varsigma_*$. Then,
\begin{align*}
\lambda_T \lesssim \sqrt{\Theta} r_0 + \gamma \sqrt{T} \varsigma_*+ \gamma \sqrt{T} \delta + \gamma T \delta  \lesssim
 \sqrt{\Theta} r_0+ \gamma T \delta \lesssim \sqrt{\Theta} r_0.
\end{align*}
This implies that
\[
\lambda_T \delta \lesssim  \upsilon,\qquad \rho_s \lambda_T (\ov M + D \ov N) \lesssim \upsilon.
\]
Combining the results above yield   in Theorem~\ref{theo:a-sum1},
\[
\mr_T(\gamma, \ov s)  \lesssim \frac{\Theta r_0^2}{\gamma T} + \gamma \varsigma_*^2 + \upsilon \lesssim \upsilon
\]
where we used again that $\gamma \leq \upsilon / \varsigma_*^2$.

\begin{comment}
Specifically, we need to verify that, under the premises of the proposition,
\be
\rho_{\ovs(\upsilon)}\lambda_{T_*(\upsilon)}(\ov M+D\ov N)\lesssim \upsilon
\ee{rhoupy}
and
\be
\delta\lambda_{T_*(\upsilon)}(1+\gamma^2\ovs(\upsilon)^2\ell^2\cL^2(1+\kappa))\lesssim \upsilon.
\ee{deltaupy}
Observe that $\gamma\varsigma_*\sqrt{T}\leq r_0\sqrt{\Theta}$ due to $\gamma\leq \frac{r_0}{\varsigma_*}\sqrt{\frac{\Theta}T}$. We also have $\upsilon\gamma T\leq \Theta r_0^2$. Indeed, when
$T_*(\upsilon)\lesssim \upsilon^{-1}\Theta\cL r_0^2\ov s(\upsilon)\ell\sqrt{1+\kappa}$, one has
$\gamma T\lesssim \upsilon^{-1}\Theta r_0^2$ due to $\gamma\leq  \left[\cL\ovs(\upsilon)\ell\sqrt{1+\kappa}\right]^{-1}$.
Similarly, $\gamma T\lesssim \upsilon^{-1}\Theta r_0^2$ for $T_*(\upsilon)\lesssim \upsilon^{-3/2}\Theta\varsigma_*r_0^2\ov s(\upsilon)\ell\sqrt{\cL}$ because
$\gamma \leq \left(\frac\Theta{T\cL}\right)^{1/3} \left(\frac{r_0}{\ovs(\upsilon)\ell\varsigma_*}\right)^{2/3}$, while $\gamma T\lesssim   \upsilon^{-1} \Theta r_0^2$
for
$T_*(\upsilon)\lesssim \frac{\Theta\varsigma_*^2r_0^2}{\upsilon^2}$, due to $\gamma\leq \frac{r_0}{\varsigma_*}\sqrt{\frac{\Theta}{T}}$.
Thus we have
(cf. \rf{lamT.b})
\[\lambda_{T}\lesssim (r_0\sqrt{\Theta}+\gamma\varsigma_*\sqrt{T})+\delta\gamma T\lesssim r_0\sqrt{\Theta}+\delta \Theta r_0^2\lesssim r_0\sqrt{\Theta}.
\]
We conclude that \rf{rhoupy} holds by definition of $\ovs(\upsilon)$; \rf{deltaupy} is also verified due to $\gamma\leq  \left[\cL\ovs(\upsilon)\ell\sqrt{1+\kappa}\right]^{-1}$.\qed
\end{comment}

\subsection{Proof of Proposition \ref{prop:SR}}
\label{sec:proof:SMDSR}
{\bf 1$^o$.} Note that under \ref{hyp:sparse1} the following result holds.
\begin{lemma}\label{lem:t001}\cite[Lemma A.1]{juditsky2023sparse}
Let $x_*\in \Xset \subset E$ be $\chi$-sparse,  $x\in \Xset $, and let $x_\chi \eqdef \sparse_\chi(x)$---an optimal solution to \rf{sparseapp}. We have
\be
\norm{x_\chi-x_*}\leq \sqrt{2\chi}\norm{x_\chi-x_*}_2\leq 2\sqrt{2 \chi}\norm{x-x_*}_2.
\ee{sparb0}
\end{lemma}
We have $\|\ov y_{-1}-x_*\|\leq R$. Let us assume that $ \bE\left[\|\ov y_{k-1}-x_*\|^2 \right]\leq 2^{-k} R^2$  for some $0\leq k\leq K$.
By Theorem \ref{theo:a-sum1}, bound \rf{muk} implies that
\[
	    \bE[ f(\ov x_{T_k -s_k})] - f_* \leq\mr_{T_k}(\gamma_k,s_k)\leq 2^{-k} \frac{\underline\varkappa}{32\bar\chi}R^2\leq
 2^{-(k+1)} \frac{\underline\varkappa}{16\bar\chi}R^2.
 \]
 By the quadratic minoration assumption \ref{hyp:sparse2},
 \[\bE[\|\ov x_{T_k-s_k}-x_*\|_2^2]\leq{2}\underline\varkappa^{-1}\left(\bE[ f(\ov x_{T_k-s_k})] - f_*\right) \leq 2^{-(k+1)} (8\bar\chi)^{-1}R^2,
 \]
 which implies, by Lemma~\ref{lem:t001}, for $\ov y_k \eqdef \sparse_{\bar\chi}( \ov x_{T_k -s_k})$
 \[\bE[\norm{\ov y_k-x_*}^2]\leq 8\bar\chi\bE[\norm{\ov x_{T_k - s_k}-x_*}_2^2]\leq 2^{-(k+1)} R^2.
 \]
 \paragraph{2$^o$.} Let us fix the index $k\geq 0$ of the stage of the algorithm. Let us check the premises of Proposition~\ref{cor:complex1}. Notice that
 \[
 \upsilon_k =\frac{\underline\varkappa}{32\bar\chi}r_0^2\leq \half\Theta\cL r_0^2
 \]
 due to $\Theta\geq \half$ and $\underline\varkappa\leq \cL$. Next, recall that by definition of $\ov s_k$ we have (cf.\rf{ovsk})
 \be
 \rho_{\ov s_k}\leq 2^{-k/2}c \frac{\ul\varkappa R}{(\ov M+D\ov N) \sqrt{\Theta} \bar\chi},
 \ee{ovs_k}
implying that
\[
 \rho_{\ov s_k}\leq \frac{c' \upsilon_k}{(\ov M + D \ov N) \sqrt{\Theta} 2^{-k/2} R}
\]
for some absolute constant $c'$. Finally, \eqref{deltau} holds since by \ref{hyp:sparse3}, $\delta =0$. Hence, by Proposition \ref{cor:complex1}, the corresponding ``stage length" $T_k$ can be bounded, up to an absolute constant, with
\begin{align*}
T^*_k
&=\max\left\{\frac{\Theta\ovs_k\ell\cL \sqrt{1+\kappa} \bar\chi}{\ul{\varkappa}},\,2^{k/2}\frac{\Theta\ovs_k\ell\sqrt{\cL}\varsigma_*\bar\chi^{3/2}}{\ul\varkappa^{3/2} R},\;
2^k\frac{\Theta\varsigma_*^2\bar\chi^2}{\ul\varkappa^2 R^2}
\right\}.
\end{align*}
To bound the total complexity $T(\ov\upsilon)$ of the algorithm, it remains to sum $T^*_k$ over the number $K\leq
K_*(\ov\upsilon)$ of algorithm stages required to attain the expected accuracy $\ov\upsilon^2$ which yields
\[
T_*(\ov\upsilon)= \max\left\{\frac{\Theta\ell\cL \sqrt{1+\kappa}\bar\chi}{\ul{\varkappa}}\sum_{k\leq K_*(\ov\upsilon)}\ovs_k,\;
\frac{\Theta\ell\sqrt{\cL}\varsigma_*\bar\chi^{3/2}}{\ul\varkappa^{3/2} R}\sum_{k\leq K_*(\ov\upsilon)}2^{k/2}\ovs_k,\;\frac{\Theta\varsigma_*^2\bar\chi^2}{\ul\varkappa^2 \ov\upsilon^2}\right\}
\]
(recall that $2^{K_*(\ov \upsilon)}\lesssim  R^2/\ov\upsilon^2$).\qed

\subsection{Proof of Corollary \ref{cor:SR}}
\label{sec:proof:cor:SR}
In the case of geometrically decreasing  $\rho_s$, we have, cf. \rf{ovs_k},
\[
\ovs_k\lesssim s_*\left(k+\ln\left({(\ov M+D\ov N)\sqrt{\Theta}\bar\chi}/{(\ul\varkappa R)}\right)\right).
\]
%with an absolute constant $c''_s$.
We can now evaluate the corresponding bound \rf{Tov}. Observe that
\begin{align*}\sum_{k=0}^{K_*(\ov\upsilon)}\ovs_k&\lesssim s_*\left(K_*(\ov\upsilon)^2+K_*(\ov\upsilon)\ln\left[{\ul\varkappa^{-1} R^{-1}}(\ov M+D\ov N)\sqrt{\Theta}\bar\chi\right]\right) \\
&\lesssim s_*\left(\ln\left[{R}/{\ov\upsilon}\right]+\ln\left[{(\ov M+D\ov N)\sqrt{\Theta}\bar\chi} / {(\ul\varkappa R)}\right]\right)^2
\lesssim s_*\ln^2\left[{(\ov M+D\ov N)   \sqrt{\Theta}\bar\chi}/({\ul\varkappa \ov\upsilon})\right]
\end{align*}
and
\begin{align*}\sum_{k=0}^{K_*(\ov\upsilon)}2^{k/2}\ovs_k&\lesssim s_*2^{K_*(\ov\upsilon)/2}\left(K_*(\ov\upsilon)+\ln\left[{(\ov M+D\ov N)\sqrt{\Theta}\bar\chi}/{(\ul\varkappa R)}\right]\right)
\\&\lesssim \frac{s_*R}{\ov\upsilon}\ln\left[{(\ov M+D\ov N)\sqrt{\Theta}\bar\chi}/{(\ul\varkappa \ov\upsilon)}\right].
\end{align*}
When substituting the above bounds into \rf{Tov}, we arrive at
\[
\ov T_*(\ov\upsilon)\leq \max\left\{\frac{\Theta\ell \cL \sqrt{1+\kappa} \bar\chi s_*\nu^2(\ov\upsilon)}{\ul{\varkappa}},\;
\frac{\Theta \ell \sqrt{\cL} \varsigma_*\bar\chi^{3/2}s_*\nu(\ov\upsilon)}{\ul\varkappa^{3/2} \ov\upsilon},\;\frac{\Theta\varsigma_*^2\bar\chi^2}{\ul\varkappa^2 \ov\upsilon^2}\right\}.
\eqed
\]

\section{Proof of Lemma \ref{lem:vfar0}}\label{sec:proof:vfar0}
We use the standard shorthand notation $\P^k \eqdef \underbrace{\P\circ ... \circ \P}_{k\mbox{ times}}$,
\[
\nu f \eqdef \int f(u)\nu(du), \quad %\P f(v)=\int f(u)\P(v,du)=\bE[f(\eta_{t+1}|\eta_t=v],\quad
\P^k f(v) \eqdef \int f(u)\P^k(v,du)=\bE[f(\eta_{t+k})|\eta_t=v].
\]
\paragraph{Case \ref{vfar:r:item2}.}  We put $V(u) \eqdef \exp( \normA{u})$. We apply Lemma~\ref{lem:vfar0:density} below with $h(u) \eqdef \Phi_i(u) \Phi_j(u)$. Under \ref{vfar:phi}, $\sup_u |h(u)|\leq \ov \phi^2$; this yields $\norm{h}_{\sqrt{V}} \leq \ov \phi^2$. \Cref{lem:vfar0:density}  along with $\bE\left[ \exp(\normA{\eta_0}) \right]<\infty $ (see \ref{vfar:zeta}) imply that \rf{new7} holds with geometrically decreasing $\rho_s$ and random $\sN_t$ proportional to $\exp(\normA{\eta_t}/2)$ such that $\sup_t \bE\left[ \sN_t^2\right]<\infty$.
\paragraph{Case \ref{vfar:r:item1}.} Let $V(u) \eqdef 2 \exp( \normA{u})$ and $W(u,u') \eqdef (V(u) + V(u'))/2$.
We set
\[
|h|_{\alpha} \eqdef \sup_{u\neq u'} \frac{h(u) - h(u')}{\normA{u - u'}^{1-\alpha} W^\alpha(u,u')}.
\]
We apply Lemma~\ref{lem:vfar} below with $h(u) \eqdef \Phi_i(u) \Phi_j(u)$ and $\alpha=1/2$. Under \ref{vfar:phi}-item~\ref{vfar:phi:item1}, %for $\cL=2L \ov \phi$
one has \[
|h(u) - h(u') | \leq \min\left\{2L\ov\phi \normA{u -u'},2\ov\phi^2\right\}
\leq 2\ov\phi^{3/2}(L\normA{u -u'})^{1/2}.
\] Hence,
\[
\sup_{u\neq u'} \frac{h(u)-h(u')}{\normA{u - u'}^{1/2}} \leq 2\ov\phi^{3/2} \sqrt{L}, \qquad \sup_{u\ne u'} \frac{h(u)-h(u')}{\big( \normA{u - u'} W(u,u')\big)^{1/2}}  \leq \ov\phi^{3/2} \sqrt{2L}.
\]
\Cref{lem:vfar} and condition $\bE\left[ \exp(\normA{\eta_0}) \right]<\infty $ (c.f. \ref{vfar:zeta})  imply that \rf{new7} is satisfied with geometrically decreasing rate $\rho_s$ and $\sN_t$ proportional to $\exp(\normA{\eta_t}/2)$ such that $\sup_t \bE\left[ \sN_t^2\right]<\infty$.
\qed

\begin{lemma} \label{lem:vfar0:density}
Consider the Markov chain \rf{mc0} and assume that \ref{vfar:zeta} and \ref{vfar:r:item2} hold. Set $V(u) \eqdef \exp( \normA{u})$. Then the chain $\{\eta_t,t  \geq 0\}$ is $\Psi$-irreducible, aperiodic, and possesses a unique invariant distribution $\pi$ such that $\pi(V)<\infty$. In addition, there exist $\rho \in \ooint{0,1}$ and a constant $C$ such that for any measurable function $h: \bR^p \to \bR$ such that $\norm{h}_{\sqrt{V}} \eqdef \sup_{\bR^p}|h(\cdot)|/\sqrt{V(\cdot)} < \infty$ and any initial distribution $\nu$ such that $\nu(\sqrt{V}) < \infty$
    \[
    \left| \nu \P^t h - \pi(h) \right| \leq C \norm{h}_{\sqrt{V}} \, \rho^t  \, \nu(\sqrt{V}).
    \]
    Finally, $\sup_t \bE\left[V(\eta_t) \right]< \infty$ when $\eta_0 \sim \nu$ for any initial distribution $\nu$ such that $\int V(u) \, \nu(\rmd u)< \infty$.
\end{lemma}
\noindent{\bf Proof}    Let us denote by $\P$ the transition kernel of the Markov chain $\{\eta_t,t \geq 0\}$, and let $\mathsf{p}$ be the density of $\zeta_1$:
    \[
    \P(u,\A) \eqdef \mathbb{P}\left( r(u) +\zeta_1 \in \A \right) = \int_\A \mathsf{p}(z -r(u))  \,  \rmd z.
    \]
 \paragraph{\bf 1$^o$.} [$\Psi$-irreducibility] Let $u \in \bR^p$ and $\A$ be a Borel set with positive Lebesgue measure. Since $\mathsf{p}$ is positive on $\bR^p$ by  \ref{vfar:r:item2}, then
    \[
    \P(u,\A) =\int_\A \mathsf{p}(z-r(u)) \, \rmd z > 0.
    \]
    Hence, by \cite[Proposition 4.2.1]{meyn:tweedie:2009}, the chain is  Lebesgue-irreducible and any compact set with positive Lebesgue measure has a positive measure under any maximal irreducibility measure by \cite[Proposition 4.2.2.]{meyn:tweedie:2009}.

\paragraph{\bf 2$^o$.} [Compact sets are one-small]  Let $\mathsf{K}$ be a compact subset of $\bR^p$. For any Borel set $\A$ and $u \in \mathsf{K}$, it holds
    \[
    \P(u,\A) \geq \int_{\A \cap \mathsf{K}} \mathsf{p}(z-r(u)) \, \rmd z \geq \epsilon \, \int_{\A \cap \mathsf{K}}\, \rmd z.
    \]
   Due to the continuity of $r$, $r(u)$ takes values in a compact subset of $\bR^p$ when $u \in \mathsf{K}$; now \ref{vfar:r:item2} implies that $\epsilon >0$.

%\paragraph{\bf 3$^o$.} [$V$ is unbounded off petite set]  Since the compact sets are one-small, the function $V$ is unbounded off petite set (see e.g. \cite[Section 8.4.2.]{meyn:tweedie:2009}).
%
%\paragraph{\bf 4$^o$.} [The chain is aperiodic] Any compact set  is one-small, with Lebesgue measure as minorizing measure: this implies that the chain is strongly aperiodic (see e.g. \cite[page 114]{meyn:tweedie:2009}).
\paragraph{\bf 3$^o$.} [The chain is aperiodic]  Since the compact sets are one-small, the function $V$ is unbounded off petite set (see e.g. \cite[Section 8.4.2.]{meyn:tweedie:2009}).
This also implies that the chain is strongly aperiodic (see e.g. \cite[page 114]{meyn:tweedie:2009}).

\paragraph{\bf 4$^o$.} [Drift inequalities]
Let $u \in \bR^p$. \begin{align*}
\P V(u) & =     \, \bE\left[ \exp( \normA{ r(u) + \zeta_1}) \right] \leq    \,  \exp( \normA{ r(u)}) \,  \bE\left[ \exp( \normA{\zeta_1}) \right]     \\
& \leq   \bE\left[ \exp( \normA{\zeta_1}) \right]  \,   \exp\left( \normA{r(u)} - \normA{u}) \right) \,  V(u).
\end{align*}
Under \ref{vfar:zeta}, there are $\lambda \in \ooint{0,1}$ and $M_\lambda>0$ such that for any $\normA{u } >M_\lambda$, $\P V(u) \leq \lambda V(u)$. On the other hand, by \ref{vfar:r:item2}, $r$ is bounded on any compact set, hence, for all $u$ such that $\normA{u} \leq M_\lambda$, $\P V(u) \leq b_\lambda$  for some  $b_\lambda<\infty$, implying that
\be
\P V(u) \leq \lambda V(u) + b_\lambda 1_{\{\normA{u} \leq M_\lambda \}}(u).
\ee{eq:vfar2:driftLyapunov}
By iterating this inequality, we have $\P^t V(u) \leq \lambda^t V(u) + b_\lambda / (1-\lambda)$ for any $t \geq 1$, from which we obtain  that $\sup_t \bE\left[V(\eta_t) \right]  = \sup_t \nu \P^t V < \infty$ when $\eta_0 \sim \nu$, for any initial distribution $\nu$ such that $\nu(V) < \infty$.

Jensen's inequality together with \eqref{eq:vfar2:driftLyapunov} yield
\[
\P \sqrt{V}(u) \leq \sqrt{\lambda}  \sqrt{V}(u) +  \sqrt{b_\lambda} 1_{\{\normA{u} \leq M_\lambda \}}(u).
\]
By \cite[Theorem 14.0.1., Lemma 15.2.8. and Theorem 16.0.1.]{meyn:tweedie:2009}, the chain $\{\eta_t, t \geq 0 \}$ possesses a unique invariant distribution $\pi$ such that $\pi(V)< \infty$; in addition,  there exist $\rho \in \ooint{0,1}$ and $C$ such that for any function $h: \bR^p \to \bR$ such that $ \norm{h}_{\sqrt{V}} \eqdef \sup_{u \in \bR^p} |h(u)|/\sqrt{V}(u) < \infty$:
\[
\forall t\geq 1, \forall u \in \bR^p,  \qquad \left| \P^t h(u) - \pi(h)\right| \leq C \norm{h}_{\sqrt{V}} \rho^t \sqrt{V(u)}.\eqed
\]

\begin{lemma} \label{lem:vfar}
Consider the Markov chain \rf{mc0} and assume that \ref{vfar:zeta} and \ref{vfar:r:item1} hold. Set $V(u) \eqdef 2 \exp( \normA{u})$.
Then, $\sup_t \bE\left[ V(\eta_t) \right] < \infty$  for any initial distribution $\eta_0\sim\nu$ of the chain
such that $\bE_{\eta\sim\nu}[V(\eta)] < \infty$.

In addition, the chain has a unique invariant distribution $\pi$ which satisfies $\bE_{\eta\sim\pi}[V(\eta)] < \infty$.

Furthermore, for any $\alpha \in \ooint{0,1}$, there exist $\rho \in (0,1)$ and a constant $C$ such that for all initial distribution $\nu$, all $s \geq 0$ and any measurable function $h$ with finite $|h|_\alpha$ (notation: $h\in \mathrm{Lip}_\alpha$),
\be
\big| \bE_{\eta_0\sim\nu}[h(\eta_s)] - \bE_{\eta\sim\pi}[h(\eta)] \big| \leq C \rho^s \left( \nu(V^\alpha) + \pi(V^\alpha) \right) \, |h|_{\alpha}.
\ee{eq:FAR:geomergodic}
\end{lemma}
\noindent{\bf Proof.}
We check the assumptions of \cite[Theorem 20.4.5]{douc2018markov}.
\paragraph{\bf 1$^o$.}[Geometric drift condition]
 It holds
\begin{align*}
\P V(u) & =  2   \, \bE\left[ \exp( \normA{ r(u) + \zeta_1}) \right] \leq 2   \,  \exp( \normA{ r(u)}) \,  \bE\left[ \exp( \normA{\zeta_1}) \right]     \\
& \leq   \bE\left[ \exp( \normA{\zeta_1}) \right]  \,   \exp\left( \normA{r(u)} - \normA{u} \right) \,  V(u).
\end{align*}
Under \ref{vfar:zeta}, there are $\lambda \in \ooint{0,1}$ and $M_\lambda>0$ such that for any $\normA{u } >M_\lambda$, $\P V(u) \leq \lambda V(u)$. On the other hand, by \ref{vfar:r:item1}, $r$ is bounded on any compact set, hence, for all $u$ such that $\normA{u} \leq M_\lambda$, $\P V(u) \leq b_\lambda$  for some  $b_\lambda<\infty$, implying that
\[
\P V(u) \leq \lambda V(u) + b_\lambda 1_{\{\normA{u} \leq M_\lambda \}}(u).
\]
{\bf 2$^o$.} We consider the Wasserstein distance $W_d$ associated with the cost function $d(u,u') =\normA{u-u'}$; it is symmetric, lower semi-continuous and $d(u,u')=0$ iff $u=u'$, verifying condition A 20.1.5 in \cite{douc2018markov}.

Moreover, for all $u,u' \in \bR^p$, we have $W_d(\P(u,\cdot), \P(u',\cdot)) \leq d(u,u')$.
Indeed, consider the coupling kernel of $(\P, \P)$ defined by \[
\P_c((u, u'), \AA \times \BB) = \Prob\left( r(u) + \zeta_1 \in \AA, r(u') + \zeta_1 \in \BB\right).
\]  Under \ref{vfar:r:item1}, we have
    \[
    \int \P_c((u,u'); \rmd u_1 \rmd u_1') \normA{u_1- u_1'} = \bE\left[ \normA{ r(u) + \zeta_1 - r(u') - \zeta_1} \right] = \normA{ r(u) - r(u')}  \leq \normA{ u - u'}
    \]
    for all $u, u' \in \bR^p$; this yields $W_d(\P(u,\cdot), \P(u',\cdot)) \leq d(u,u').$

Let now $D_\lambda$ such that $\lambda + 2 b_\lambda /(1+D_\lambda)<1$, and let us show  that \[
\CC=\{u:\,V(u) \leq D_\lambda\} \times \{u':\,V(u') \leq D_\lambda \}
\] is a {\em $(d,1,\epsilon)$-contracting set.}  The level set $\{u:\,V(u) \leq D_\lambda \}$ is a compact subset of $\bR^p$ and by \ref{vfar:r:item1},  there exists $\epsilon_\lambda>0$ such that  \[
     \int \P_c\left((u,u'); \rmd u_1 \rmd u_1'\right) \normA{u_1- u_1'}  = \normA{ r(u) - r(u')}  \leq  (1- \epsilon_\lambda) \normA{ u - u'}, \qquad \forall (u, u') \in \CC.
    \]
Finally, one clearly has
\[
\mathsf{d}(u,u') =\normA{ u - u'} \leq \normA{u } + \normA{u'} \leq
(V(u) +V(u'))/2. \]

We have checked that under \ref{vfar:zeta} and \ref{vfar:r:item1}, the Markov chain $(\eta_t)$ verifies all conditions of \cite[Theorem 20.4.5]{douc2018markov}, implying that $\P$ admits a unique invariant distribution $\pi$ which verifies $\pi(V)<\infty$.   Furthermore, for any $\alpha \in (0,1)$, there exist $\rho \in (0,1)$ and $C < \infty$ such that for all initial distribution $\nu$ and all $s \geq 0$
\[
W_d(\nu \P^s, \pi) \leq W_{\alpha}(\nu \P^s, \pi) \leq C \rho^s \left( \nu(V^\alpha) + \pi(V^\alpha) \right)
\]
where $W_{\alpha}$ is the Wasserstein distance associated with the cost function $\normA{u-u'}^{1-\alpha}W^\alpha(u,u')$.
where $W(u,u') =(V(u) + V(u'))/2$. We conclude by \cite[Corollary 20.4.7]{douc2018markov} that for all measurable  $h \in \mathrm{Lip}_{\alpha}$,
\[
\left| \nu \P^s(h) - \pi(h) \right| \leq C \rho^s \left( \nu(V^\alpha) + \pi(V^\alpha) \right) \, |h|_{\alpha}.\eqed
\]

\bibliographystyle{abbrv}
\bibliography{n.bib}
\end{document}